\documentclass[twocolumn]{autart}

\usepackage{graphicx}    
\usepackage{subcaption}
\usepackage{amsmath}     
\usepackage{amssymb}     
\usepackage{url}
\usepackage{enumitem}
\usepackage{xcolor}
\usepackage[bookmarks=false,hidelinks]{hyperref}

\usepackage{comment}
\usepackage{algorithm}
\usepackage{algpseudocode}
\allowdisplaybreaks[2]

\algnewcommand{\StateCont}[1]{\Statex\hspace{\algorithmicindent}#1}

\begin{document}

\begin{frontmatter}

\title{Decentralized Contingency MPC based on Safe Sets for Nonlinear Multi-agent Collision Avoidance}

\author[Inst1]{Max Studt$^\ast$}\ead{m.studt@uni-luebeck.de}, \
\author[Inst1]{Georg Schildbach}\ead{georg.schildbach@uni-luebeck.de}
\address[Inst1]{Institute for Electrical Engineering in Medicine, University of Luebeck, Luebeck, Germany}

\begin{keyword}
Decentralized MPC; Contingency MPC; Nonlinear systems; Multi-agent systems; Collision avoidance;
Recursive feasibility; Safe sets; Lyapunov constraint
\end{keyword}

\begin{abstract}
Decentralized collision avoidance remains challenging, particularly when agents do not communicate any information related to planned trajectories. Most existing approaches either rely on conservative coordination mechanisms or provide limited guarantees on recursive feasibility and convergence. This paper develops a decentralized contingency MPC framework for multi-agent systems with nonlinear dynamics that achieves collision-free motion under a state-only information pattern. Each agent follows the same consensual rule set, enabling safe decentralized planning without communication. Each agent solves a local optimization problem that couples a nominal trajectory with a contingency certificate ensuring a feasible backup maneuver under receding-horizon operation. A novel geometric and decentralized safe-set update mechanism prevents feasibility loss between consecutive time steps. The resulting scheme guarantees recursive feasibility, including collision avoidance, and establishes a Lyapunov-type convergence result to an admissible safe equilibrium. Simulation results demonstrate performance in both sparse and dense multi-agent environments, including cluttered bottleneck scenarios and under plug-and-play operation.
\end{abstract}

\end{frontmatter}

\theoremstyle{plain}
\newtheorem{theorem}{Theorem}
\newtheorem{lemma}{Lemma}
\newtheorem{proposition}{Proposition}

\theoremstyle{definition}
\newtheorem{assumption}{Assumption}
\newtheorem{remark}{Remark}
\newtheorem{definition}{Definition}
\newtheorem{corollary}{Corollary}
\newenvironment{proof}
  {\par\noindent\textbf{\textit{Proof: }}\itshape}
  {\hfill$\square$\par}

\section{Introduction}
\vspace*{-0.1cm}

Multi-agent navigation and control problems require \emph{hard} collision avoidance while simultaneously achieving mission objectives such as reference tracking. Model Predictive Control (MPC) provides a natural framework for this setting, since constraints can be handled explicitly while optimizing performance over a receding horizon~\cite{Mayne2000}.

\vspace*{-0.1cm}
\subsection{Literature overview}
\vspace*{-0.1cm}
\noindent\textit{Centralized planning} \vspace{0.15cm} \\
A first class of approaches enforces collision avoidance in a \emph{centralized} manner by solving a coupled optimization problem for the entire fleet. Such formulations can encode global constraints directly, but typically scale poorly and require a coordinator with access to global information. Mixed-integer linear programs (MILPs) are a prominent example, enabling explicit collision-avoidance logic at the price of combinatorial complexity~\cite{RichardsSchouwenaarsHowFeron2002,Schouwenaars2001ECC}. Receding-horizon MILP/MPC variants have also been demonstrated in aerospace guidance applications, highlighting both the potential and computational burden of centralized approaches~\cite{SchouwenaarsValentiFeronHow2005}.

\vspace*{-0.1cm}
\noindent\textit{Distributed / decentralized MPC with coordination} \vspace{0.15cm} \\
To improve scalability, many distributed or decentralized MPC schemes let agents solve local problems while coordinating to handle coupling constraints~\cite{Negenborn2009,Scattolini2009}. This coordination may rely on repeated communication rounds~\cite{CommunicationRounds}, negotiation of predicted trajectories~\cite{Negotiation}, consensus/agreement protocols~\cite{Lunze2022}, or distributed optimization methods such as ADMM/ALADIN~\cite{Aladin}. Related non-iterative approaches exchange compact descriptions of neighbor behavior, e.g., contracts of guaranteed future coupling-variable trajectories, which can support recursive feasibility and stability while limiting communication to local neighbors~\cite{LUCIA2015205}. However, such schemes still rely on communication and may introduce latency or coordination assumptions that are undesirable in safety-critical settings~\cite{KOHLER20191,Stewart2010,Venkat2007}.
An intermediate regime considers limited or time-varying communication, including bounded ranges, changing neighbor sets, and plug-and-play operation. Recent MPC formulations address this by reasoning about topology changes over the prediction horizon, for example, via multi-trajectory predictions that couple a nominal trajectory with a worst-case topology-change trajectory~\cite{Saccani2023}. More broadly, plug-and-play concepts have been studied in networked control and MPC for systems with changing components~\cite{LUCIA2015205,Riverso2013,Stoustrup2009}.
Several MPC-based approaches further reduce coordination effort toward single exchanges or local sensing. Examples include decentralized cooperative and reactive MPC-style navigation in unknown environments~\cite{HoyMatveevSavkin2012}, multi-robot NMPC with feasibility and stability analyses~\cite{LafmejaniBerman2021}, and decentralized deadlock-prevention layers that trigger coordination maneuvers in dense settings~\cite{MerryGoRound}.

\vspace*{-0.1cm}
\noindent\textit{Decentralized approaches without communication} \vspace{0.15cm} \\
Geometric methods such as velocity obstacles~\cite{FioriniShiller1998} and ORCA~\cite{vanDenBergORCA2011} provide decentralized collision avoidance, including extensions to nonlinear dynamics~\cite{Bareiss2015} and combinations with flatness-based MPC~\cite{Arul2019DCAD}. Buffered Voronoi-based predictive approaches
offer feasibility guarantees for single-integrator dynamics~\cite{Zhou2017} or stochastic linear systems~\cite{Vinod2024}, and may require only local communication~\cite{Chen2024}. In contrast, the presented approach establishes recursive feasibility, collision avoidance, and Lyapunov-type convergence for nonlinear agent dynamics under the main assumption of state observations, without exchanging planned trajectories or control inputs.

\vspace*{-0.1cm}
\noindent\textit{Safety filters and barrier-function based safety} \vspace{0.15cm} \\
Optimization-based safety filters and control barrier functions enforce safety by rendering a safe set forward invariant via online correction problems~\cite{Ames2019}. Predictive safety filters connect this idea to MPC by embedding constraint satisfaction into a predictive structure~\cite{Wabersich2021}.
Vinod et al.~\cite{Vinod2024} combine decentralized predictive safety filtering with buffered Voronoi cells for stochastic linear systems. The focus is on nonlinear nominal dynamics and the guarantees of recursive feasibility, probabilistic collision avoidance, and Lyapunov-type convergence under state-only observations.

\vspace*{-0.1cm}
\subsection{Contributions}
\vspace*{-0.1cm}

To address these gaps, this paper builds on Contingency MPC, 
which couples a nominal plan with a
safety or backup plan, ensuring the existence of an
emergency maneuver at all times~\cite{Alsterda2019}. A related communication-free
multi-agent collision-avoidance setting has been studied for
\emph{linear} systems using pairwise bisecting-plane separation
constraints and a deterministic backup evolution~\cite{Georg}.
In contrast, this paper develops a decentralized contingency
MPC framework for \emph{nonlinear} multi-agent systems under
state-only information.\\
The central contribution is the \textit{freeze-or-shift}
(\textit{FoS}) update of local safe sets. Each contingency
plan remains within its own active safe set, avoiding explicit
pairwise collision-avoidance constraints on contingency
trajectories. The FoS rule maintains disjointness, while
tail-containment constraints ensure feasibility of shifted
contingency plans.
Together, these mechanisms guarantee recursive feasibility
and collision avoidance under the stated assumptions without
exchanging predicted trajectories or control inputs.
A contingency-cost decrease constraint implies
convergence to an admissible equilibrium, once the
terminal equilibrium becomes constant.
The main contributions of this paper are:\vspace*{-0.3cm}
\begin{itemize}
  \item[C1] A geometric formulation of decentralized
  collision avoidance through disjoint local safe sets,
  without explicit pairwise collision-avoidance constraints
  on the contingency trajectories.

  \item[C2] A fully decentralized \textit{FoS} update rule
  that preserves safe-set disjointness and, together with
  tail-containment constraints, maintains feasible
  contingency plans under safe-set updates.

  \item[C3] A decentralized contingency MPC framework
  for nonlinear multi-agent systems under a state-only
  information pattern, with constructive proofs of
  recursive feasibility and collision avoidance.

  \item[C4] A Lyapunov-type decrease result for the
  contingency cost, implying convergence to a fixed
  admissible equilibrium when the selected terminal
  equilibrium becomes constant.

  \item[C5] A nonlinear simulation study illustrating
  the proposed framework in sparse, dense, cluttered
  bottleneck, and plug-and-play multi-agent scenarios.
\end{itemize}

\vspace*{-0.1cm}
\section{Preliminaries and problem statement}
\label{sec:prelim}
\vspace*{-0.1cm}

\subsection{Notation}
\label{subsec:notation}
\vspace*{-0.1cm}
The sets of real and integer numbers are denoted $\mathbb{R}$ and $\mathbb{Z}$, respectively. $\mathbb{R}_+$ and $\mathbb{Z}_+$ denote the sets of nonnegative real and nonnegative integer numbers, respectively, and $\mathbb{Z}_{>0}$ is the set of positive integers. For integers $a,b\in\mathbb{Z}$ with $a\le b$,
\(
\mathbb{Z}_a^b := \{k\in\mathbb{Z}\mid a \le k \le b\}.
\)
is the (inclusive) index set.
The set of agents is
\(
\mathcal{I} := \{1,\dots,M\}.
\)
Discrete time is indexed by $t\in\mathbb{Z}_+$ with sampling time $T_s>0$. Continuous time is denoted $\tau\in\mathbb{R}_+$.
The next discrete time instant is denoted $t^+ := t+1$, and the previous one $t^- := t-1$.
For a vector $x\in\mathbb{R}^n$, $\|x\|$ represents the Euclidean norm.
Consider a position space of dimension $n_p \in \mathbb{Z}_{>0}$. For $c\in\mathbb{R}^{n_p}$ and $R\ge 0$, the closed Euclidean ball centered at $c$ with radius $R$ is defined as
\(
\mathbb{B}(c,R) := \{p\in\mathbb{R}^{n_p} : \|p-c\| \le R\}.
\)
The MPC prediction index $(k|t)$ is used throughout: $x_{i,(k|t)}$ denotes the prediction of $x_i(t+k)$ for agent $i$ computed at time $t$.
An equilibrium of agent $i$ is denoted $(\bar x_i,\bar u_i)$; it satisfies
\(
\bar x_i = f_i(\bar x_i,\bar u_i).
\)
When needed, the associated equilibrium position is denoted by $\bar p_i := C_i \bar x_i$.
For the contingency plan, terminal safe equilibria are denoted $(\bar x_i^{\mathrm c}, \bar u_i^{\mathrm c})$.

\vspace*{-0.1cm}
\subsection{Agent Dynamics and Constraints}
\label{subsec:dynamics}
\vspace*{-0.1cm}
Let \(\mathcal W \subseteq \mathbb{R}^{n_p}\) denote the workspace. For each agent \(i \in \mathcal{I}\), define the
body-feasible position set
\[
\mathcal W_i := \{p_i(t) \in \mathbb{R}^{n_p} : \mathbb{B}(p_i(t),r_i) \subseteq \mathcal W\}.
\]
Thus, \(p_i(t) \in \mathcal W_i\) means that the full body of agent \(i\), modeled as a
closed Euclidean ball of radius \(r_i\) centered at \(p_i(t)\), is contained in the workspace.
For unconstrained workspaces, one may use
\(\mathcal W= \mathcal W_i =\mathbb{R}^{n_p}\).

\vspace*{-0.3cm}
Agent \(i \in \mathcal{I}\) follows the discrete-time nonlinear
dynamics
\begin{equation}
x_i(t^+) = f_i(x_i(t),u_i(t)),
\label{eq:agent_dynamics}
\end{equation}
with state \(x_i(t) \in \mathbb{R}^{n_i}\) and input
\(u_i(t) \in \mathbb{R}^{m_i}\). The position of agent \(i\) is given by
\(
p_i(t) = C_i x_i(t) \in \mathbb{R}^{n_p},
\)
where \(C_i \in \mathbb{R}^{n_p \times n_i}\) extracts the position
component of the state.
The sets of \emph{admissible states} and \emph{admissible inputs} for agent \(i\) are denoted by
\(
\mathcal X_i \subset \mathbb{R}^{n_i}, \ \mathcal U_i \subset \mathbb{R}^{m_i}.
\)
The closed-loop system trajectory is required to satisfy
\(
x_i(t) \in \mathcal X_i, \ u_i(t) \in \mathcal U_i.
\)
for all $t \in \mathbb{Z}_+$. It is assumed that the state constraints are chosen such that $p_i(t) \in \mathcal W_i$ whenever $x_i(t) \in \mathcal X_i$.

\vspace*{-0.3cm}
Each agent \(i\) is assigned an individual global reference state
\(
x_i^{\mathrm{ref}} \in \mathbb{R}^{n_i},
\)
with corresponding reference position
\(
p_i^{\mathrm{ref}} = C_i x_i^{\mathrm{ref}} \in \mathcal W_i.
\)
Throughout the paper, \(x_i^{\mathrm{ref}}\) is assumed to be an
\emph{admissible equilibrium state} for agent \(i\), i.e., there exists an
admissible input \(u_i^{\mathrm{ref}} \in \mathcal U_i\) such that
\[
x_i^{\mathrm{ref}}
=
f_i(x_i^{\mathrm{ref}},u_i^{\mathrm{ref}}).
\]

\vspace*{-0.5cm}
\subsection{Collision Model and Safety Objective}
\label{subsec:collision}
\vspace*{-0.1cm}
Denote the closed Euclidean ball occupied by agent $i$ as
\(
\mathcal B_i(t) := \mathbb{B}(p_i(t),r_i).
\)
A collision between two agents $i\neq j$ at time $t$ occurs if
\(
\mathcal B_i(t)\cap \mathcal B_j(t)\neq\emptyset,
\)
equivalently if $\|p_i(t)-p_j(t)\| < r_i+r_j$.
The safety objective is collision-free closed-loop execution,
\[
\mathcal B_i(t)\cap \mathcal B_j(t)=\emptyset, \qquad \forall i\neq j,\ \forall t\in\mathbb{Z}_+,
\]
while driving each agent towards its reference or, if necessary, towards a safe admissible equilibrium.

\vspace*{-0.2cm}
\begin{assumption}[Exact nominal dynamics]
\label{ass:exact_dynamics}
The closed-loop evolution of each agent is exactly described by~\eqref{eq:agent_dynamics} under the input and state constraints $u_i(t)\in\mathcal U_i$ and $x_i(t)\in\mathcal X_i$, for all $t\in\mathbb{Z}_+$. In particular, no model mismatch and no further uncertainties are considered.
\end{assumption}

\vspace*{-0.2cm}
\begin{assumption}[Information pattern]
\label{ass:information_pattern}
In each time step, every agent $i\in\mathcal I$ can measure its own state $x_i(t)$ and the current
states $x_j(t)$ of all other active agents
$j\in\mathcal I\setminus\{i\}$ exactly.
No information is exchanged between agents; i.e.,
other agents' reference states,
future inputs, and predicted trajectories are
neither communicated nor available when solving
the local optimization problem.
\end{assumption}

\vspace*{-0.2cm}
\section{Local safe sets}
\label{sec:local_safe_sets}
\vspace*{-0.1cm}

This section introduces the geometric object used to enforce decentralized safety, namely the active safe set\linebreak associated with each agent. The active safe sets are chosen such that, from the current state of each agent, there exists a feasible backup maneuver that remains inside the set, reaches an admissible safe equilibrium in finite steps, and can subsequently be maintained there indefinitely.

\vspace*{-0.2cm}
At each time $t\in\mathbb{Z}_+$, each agent $i\in\mathcal{I}$ is associated with a time-varying active safe set
\begin{equation}
S_i^\ast(t) := \mathbb{B}\big(c_i(t),R_i(t)\big)\subseteq\mathbb{R}^{n_p},
\label{eq:local_safe_set}
\end{equation}
where $c_i(t)\in\mathbb{R}^{n_p}$ is the center and $R_i(t)\ge 0$ is the radius.
The active pair \((c_i(t),R_i(t))\) is determined by a deterministic
update rule based on a state-dependent safe-set generator. For each agent \(i\in\mathcal I\), let
\[
\Gamma_i:\mathcal X_i \rightarrow
\left\{
\mathbb B(c,R)\subseteq\mathbb R^{n_p}
\,:\,
c\in\mathbb R^{n_p},\ R\ge 0
\right\}
\]
denote a deterministic map that assigns to every admissible state
\(x_i\in\mathcal X_i\) a generated safe set. For a given state
\(x_i\), this generated safe set is written as
\[
\Gamma_i(x_i)
=
\mathbb B\big(c_i(x_i),R_i(x_i)\big).
\]
The active safe set \(S_i^\ast(t)\) used by the controller need not
coincide with the generated set \(\Gamma_i(x_i(t))\). Instead, the
generated set may either be accepted or rejected by the \textit{FoS}
update rule introduced in Section~\ref{sec:FreezeandshiftOperator}.
Thus, reconstructing active safe sets generally requires the current
state information together with one-step memory of the previously
active safe-set parameters.
Independent of this update rule, the footprint of the agent must satisfy
\begin{equation}
\mathcal{B}_i(t)\subseteq S_i^\ast(t),\qquad \forall t\in\mathbb{Z}_+.
\label{eq:footprint_containment}
\end{equation}

\vspace*{-0.2cm}
\begin{assumption}[Local reconstructability]
\label{ass:local_reconstructability}
The deterministic safe-set generation and update rules are commonly
known to all agents. During closed-loop operation, each agent
\(i\in\mathcal I\) can reconstruct the active safe set \(S_j^\ast(t)\)
of every other agent \(j\in\mathcal I\setminus\{i\}\) from the observed
state \(x_j(t)\) and one-step memory of the previously active safe-set
parameters \(S_j^\ast(t^-)\), without communication of planned
trajectories or control inputs. Furthermore, all agents use a common sampling clock and perform
the control and safe-set updates synchronously,
based on consistent state observations.
\end{assumption}

\vspace*{-0.1cm}
\begin{assumption}[Contingency recoverability]
\label{ass:contingency_recoverability}
For every agent $i\in\mathcal{I}$ and every time $t\in\mathbb{Z}_+$, the active safe set $S_i^\ast(t)$ is chosen such that there exists a feasible contingency maneuver starting from $x_i(t)$ that remains inside $S_i^\ast(t)$, satisfies all state and input constraints, and reaches an admissible safe equilibrium in finite time.
\end{assumption}

\vspace*{-0.1cm}
\begin{definition}[Admissible safe equilibrium]
\label{def:terminal_safe_equilibrium}
For every agent $i\in\mathcal{I}$ and every time $t\in\mathbb{Z}_+$ the state-input pair $(\bar x_i,\bar u_i)$ satisfying
\[
\bar x_i = f_i(\bar x_i,\bar u_i),
\qquad
\bar p_i := C_i\bar x_i \in S_i^\ast(t),
\]
and
\[
\mathbb{B}\big(\bar p_i,r_i\big)\subseteq S_i^\ast(t)\cap\mathcal{W}_i,
\]
represents an admissible safe equilibrium.
Moreover, once this equilibrium is reached, the agent can be kept inside $S_i^\ast(t)$ for all future times by applying the constant equilibrium input $\bar u_i$.
\end{definition}

\vspace*{-0.2cm}
Assumption~\ref{ass:contingency_recoverability} and Definition~\ref{def:terminal_safe_equilibrium} formalize the intended role of the active safe set: It supports both finite-time recovery and indefinite safe holding. In particular, the active safe set must contain a reachable admissible safe equilibrium and the agent is able to remain there indefinitely once this equilibrium has been reached.

\vspace*{-0.2cm}
\begin{remark}[Non-equilibrium-based terminal set]
\label{rem:terminal_invariant_set_extension}
The present paper considers systems for which terminal safety can be represented by an admissible safe equilibrium contained in the active safe set. For systems that do not admit such a fixed equilibrium, e.g., certain aircraft models, the same conceptual role could instead be played by a compact positively invariant terminal safe set $X_i^{\mathrm f}(t)\subseteq S_i^\ast(t)$ that is reachable in finite time and can be rendered invariant by a backup policy. 
\end{remark}

\vspace*{-0.2cm}
\begin{remark}[Model-dependent realization]
\label{rem:model_dependent_realization}
The concrete realization of
Assumption~\ref{ass:contingency_recoverability}
depends on the agent model. Local safe sets may be
constructed using closed-form stopping bounds,
reachable-set over-approximations, invariant-set
arguments, or other backup designs.
The generator may depend on the full state, including
absolute position. Translation invariance of the
dynamics is not required, but can simplify the
construction by allowing maneuver bounds to be reused
across positions, provided that the corresponding
state and input constraints remain satisfied.
\end{remark}

\vspace*{-0.1cm}
\section{Decentralized contingency MPC formulation}
\label{sec:dcmpc}
\vspace*{-0.1cm}

This section states the decentralized contingency MPC problem solved by each agent $i\in\mathcal{I}$ at time $t\in\mathbb{Z}_+$.
At each time step, a single optimization simultaneously computes
(i) a nominal trajectory for performance and reference tracking and
(ii) a contingency trajectory that guarantees the existence of a feasible
backup maneuver within the active safe set.
The two plans are coupled by enforcing a shared first control input.

\vspace*{-0.1cm}
\subsection{Decision variables and horizons}
\label{subsec:variables_horizons}
\vspace*{-0.1cm}

Two prediction horizons are introduced: a nominal horizon $N_{\mathrm{n}}\in\mathbb{Z}_+$
and a contingency horizon $N_{\mathrm{c}}\in\mathbb{Z}_+$.
The nominal horizon $N_{\mathrm{n}}$ is chosen to shape tracking performance.
The contingency horizon $N_{\mathrm{c}}$ is selected sufficiently large such that a feasible contingency maneuver to an admissible safe equilibrium can be represented within the active safe set.
Its choice depends on the agent dynamics, input and state constraints, and the construction of the local safe sets.

\vspace{-0.2cm}
At time $t$, agent $i$ optimizes the nominal and contingency
state--input sequences
\[
X_i^\bullet(t)=\{x^\bullet_{i,(k|t)}\}_{k=0}^{N_\bullet},
\quad
U_i^\bullet(t)=\{u^\bullet_{i,(k|t)}\}_{k=0}^{N_\bullet-1},
\]
where $\bullet\in\{\mathrm n,\mathrm c\}$ distinguishes
nominal and contingency trajectories. Each state
$x^\bullet_{i,(k|t)}\in\mathbb R^{n_i}$ has position
$p^\bullet_{i,(k|t)}=C_i x^\bullet_{i,(k|t)}
\in\mathbb R^{n_p}$.
In addition, the optimization includes the \emph{terminal contingency equilibrium pair} as a decision variable,
\begin{equation}
\bar{x}^{\mathrm{c}}_i(t) \in \mathcal X_i, \quad 
\bar{u}^{\mathrm{c}}_i(t) \in\mathcal{U}_i,  
\end{equation}
satisfying
\(
\bar{x}^{\mathrm{c}}_i(t) = f_i\big(\bar{x}^{\mathrm{c}}_i(t), \bar{u}^{\mathrm{c}}_i(t)\big).
\)
The closed-loop system evolves according to \eqref{eq:agent_dynamics} with the applied input given by the shared first move,
\begin{equation}
u_i(t) = u^{\mathrm{n},*}_{i,(0|t)} = u^{\mathrm{c},*}_{i,(0|t)},
\label{eq:shared_input_applied}
\end{equation}
where $(\cdot)^\ast$ denotes an optimal solution of the MPC problem at time $t$.

\vspace*{-0.1cm}
\subsection{Objective function}
\label{subsec:dcmpc_objective}
\vspace*{-0.1cm}
The local objective is kept generic, as the theoretical guarantees derived later rely on
structural properties rather than on a specific (e.g., quadratic) choice.
A typical design penalizes deviations of the nominal trajectory from a desired reference and regularizes the selection of the contingency equilibrium.
\vspace{-0.2cm}

Let $\ell_i^{\mathrm{n}}(\cdot,\cdot)$ denote a tracking-related stage cost and
$V_i^{\mathrm{n}}(\cdot)$ a terminal tracking cost. Moreover, let
$V_i^{\mathrm{c}}(\bar{x}^{\mathrm{c}}_i(t), x_i^{\mathrm{ref}})$ denote an offset cost that measures
the distance of the contingency equilibrium to the desired target state.
\begin{equation}
\begin{aligned}
J_i\big(X_i^{\mathrm{n}}(t), &U_i^{\mathrm{n}}(t), \bar{x}^{\mathrm{c}}_i(t), x_i^{\mathrm{ref}}\big)
:= \\
\sum_{k=0}^{N_{\mathrm{n}}-1}
&\ell_i^{\mathrm{n}}\!\big(x^{\mathrm{n}}_{i,(k|t)}, u^{\mathrm{n}}_{i,(k|t)}\big)
\;+\;
V_i^{\mathrm{n}}\!\big(x^{\mathrm{n}}_{i,(N_{\mathrm{n}}|t)}, x_i^{\mathrm{ref}}\big) \\
&\;+\;
\gamma\,V_i^{\mathrm{c}}\!\big(\bar{x}^{\mathrm{c}}_i(t), x_i^{\mathrm{ref}}\big).
\end{aligned}
\label{eq:generic_objective}
\end{equation}
The scalar $\gamma>0$ weights the preference for contingency equilibria that are close to the desired reference. Note that the cost term related to the selected terminal contingency equilibrium state is required in order to enforce a Lyapunov-like decrease
through an additional constraint (Section~\ref{subsec:dcmpc_convergence_constraint}).

\vspace{-0.2cm}
\begin{definition}[Optimal contingency equilibrium]
\label{def:opt_reachable_equilibrium}
In each step $t\in\mathbb{Z}_+$, an \emph{optimal reachable contingency equilibrium} minimizes the equilibrium-offset cost, i.e.,
\[
(\bar x_i^{\mathrm c,\ast}(t),\bar u_i^{\mathrm c,\ast}(t))
\in
\arg\min_{(\bar x_i^{\mathrm c},\bar u_i^{\mathrm c})\in\mathcal Z_i^{\mathrm c}(t)}
V_i^{\mathrm c}\!\big(\bar x_i^{\mathrm c},x_i^{\mathrm{ref}}\big),
\]
where $\mathcal Z_i^{\mathrm c}(t)$ denotes the set of admissible terminal contingency equilibrium pairs.
\end{definition}
\vspace{-0.2cm}
\begin{remark}[Role of the equilibrium-offset weight]
The scalar $\gamma>0$ regularizes the selection of the contingency equilibrium. In related multi-trajectory MPC formulations, it can be shown that for sufficiently large equilibrium-offset weights, the equilibrium selected by the optimizer can be made arbitrarily close to the offset-minimizing admissible equilibrium ~\cite[Lemma~1]{Saccani2023}.
\end{remark}

\vspace*{-0.1cm}
\subsection{Constraints}
\label{subsec:dcmpc_constraints}
\vspace*{-0.1cm}
Both predicted trajectories start from the measured state,
\begin{equation}
x^{\mathrm n}_{i,(0|t)}
=x^{\mathrm c}_{i,(0|t)}=x_i(t).
\label{eq:init_both}
\end{equation}
For $\bullet\in\{\mathrm n,\mathrm c\}$, the dynamics
\eqref{eq:agent_dynamics} and state/input constraints
are imposed by
\begin{align*}
x^\bullet_{i,(k+1|t)}
&=f_i\bigl(x^\bullet_{i,(k|t)},u^\bullet_{i,(k|t)}\bigr),\\
x^\bullet_{i,(k+1|t)}&\in\mathcal X_i,\quad
u^\bullet_{i,(k|t)}\in\mathcal U_i,
\qquad
\forall k\in\mathbb Z_0^{N_\bullet-1}.
\end{align*}
To ensure that the predicted contingency trajectory represents a valid backup plan, the terminal equilibrium constraint
\begin{equation}
x^{\mathrm{c}}_{i,(N_{\mathrm{c}}|t)} = \bar{x}^{\mathrm{c}}_i(t),
\qquad
\bar{x}^{\mathrm{c}}_i(t) = f_i\big(\bar{x}^{\mathrm{c}}_i(t), \bar{u}^{\mathrm{c}}_i(t)\big),
\label{eq:terminal_steady_state}
\end{equation}
is imposed.
The terminal position is required to satisfy
\[
\bar{p}^{\mathrm{c}}_i(t) := C_i\bar{x}^{\mathrm{c}}_i(t)\in S_i^\ast(t).
\]
Moreover, the terminal equilibrium must be reachable from the current state $x_i(t)$ within $N_{\mathrm{c}}$ steps by a feasible contingency trajectory that remains inside the active safe set, i.e.,
\(
\bar{x}_i^c(t) \in \mathrm{reach}_{N_c}(x_i(t)).
\)

\vspace*{-0.2cm}
\subsubsection{Local Safe-Set Constraints}
\vspace*{-0.1cm}

The entire contingency position plan is required to remain inside the current active safe set \(S_i^\ast(t)\). To account for the agent size, the constraint is imposed with a radius margin \(r_i\), i.e.,
\begin{equation}
\|p^{\mathrm{c}}_{i,(k|t)} - c_i(t)\| \le R_i(t) - r_i,
\qquad \forall k\in\mathbb{Z}_0^{N_{\mathrm{c}}}.
\label{eq:contingency_plan_in_active_safe_set}
\end{equation}
Equivalently, \(\mathbb{B}\bigl(p^{\mathrm{c}}_{i,(k|t)},r_i\bigr)\subseteq S_i^\ast(t)\) for all
\(k\in\mathbb{Z}_0^{N_{\mathrm{c}}}\).
Furthermore, for every prediction step \(k\), the remainder of the
contingency trajectory (from step \(k\) onward) is required to remain
inside the safe set generated by \(\Gamma_i\) from the predicted
contingency state \(x^{\mathrm c}_{i,(k|t)}\). Let this generated safe set
be given by
\[
\Gamma_i\bigl(x^{\mathrm c}_{i,(k|t)}\bigr)
=
\mathbb B\bigl(
c_i(x^{\mathrm c}_{i,(k|t)}),
R_i(x^{\mathrm c}_{i,(k|t)})
\bigr).
\]
Then the tail-containment constraint is
\begin{equation}
\begin{aligned}
\|p^{\mathrm c}_{i,(l|t)} - c_i(x^{\mathrm c}_{i,(k|t)})\|
\le
R_i(x^{\mathrm c}_{i,(k|t)})-r_i, \\
\qquad
\forall k\in\mathbb Z_0^{N_{\mathrm c}},
\ \forall l\in\mathbb Z_k^{N_{\mathrm c}}.
\label{eq:tail_containment_constraint}
\end{aligned}
\end{equation}
Constraint~\eqref{eq:tail_containment_constraint} enforces a receding-horizon recoverability property:
At each prediction stage \(k\), the planned tail
\[
\left\{x^{\mathrm{c}}_{i,(l|t)}\right\}_{l=k}^{N_{\mathrm{c}}}
\]
remains entirely inside the candidate safe set generated from the predicted state \(x^{\mathrm{c}}_{i,(k|t)}\).
In particular, after applying the first input, the shifted tail
\[
\left\{x^{\mathrm{c}}_{i,(l+1|t)}\right\}_{l=0}^{N_{\mathrm{c}}-1}
\]
is contained in the candidate safe set generated from the new initial condition
\(
x_i(t^+) = x^{\mathrm{c}}_{i,(1|t)}.
\)
Hence, if the active safe set at the next time step is chosen as this candidate safe set, the shifted contingency tail remains admissible with respect to the newly induced active safe set.
This property is a key ingredient for establishing recursive feasibility under receding-horizon implementation, including the case where the active safe set is updated by the \textit{FoS} rule.

Figure~\ref{fig:TailContainmentInfeasibility} illustrates why constraint~\eqref{eq:tail_containment_constraint} is necessary.
Even if the local problem is feasible at time $t$ with active safe set $S_i^\ast(t)$, recursive feasibility may be lost after applying the first input if the remaining contingency tail is not guaranteed to lie inside the safe set induced at the successor state.
\begin{figure}[h]
    \centering
    \includegraphics[width=0.4\textwidth]{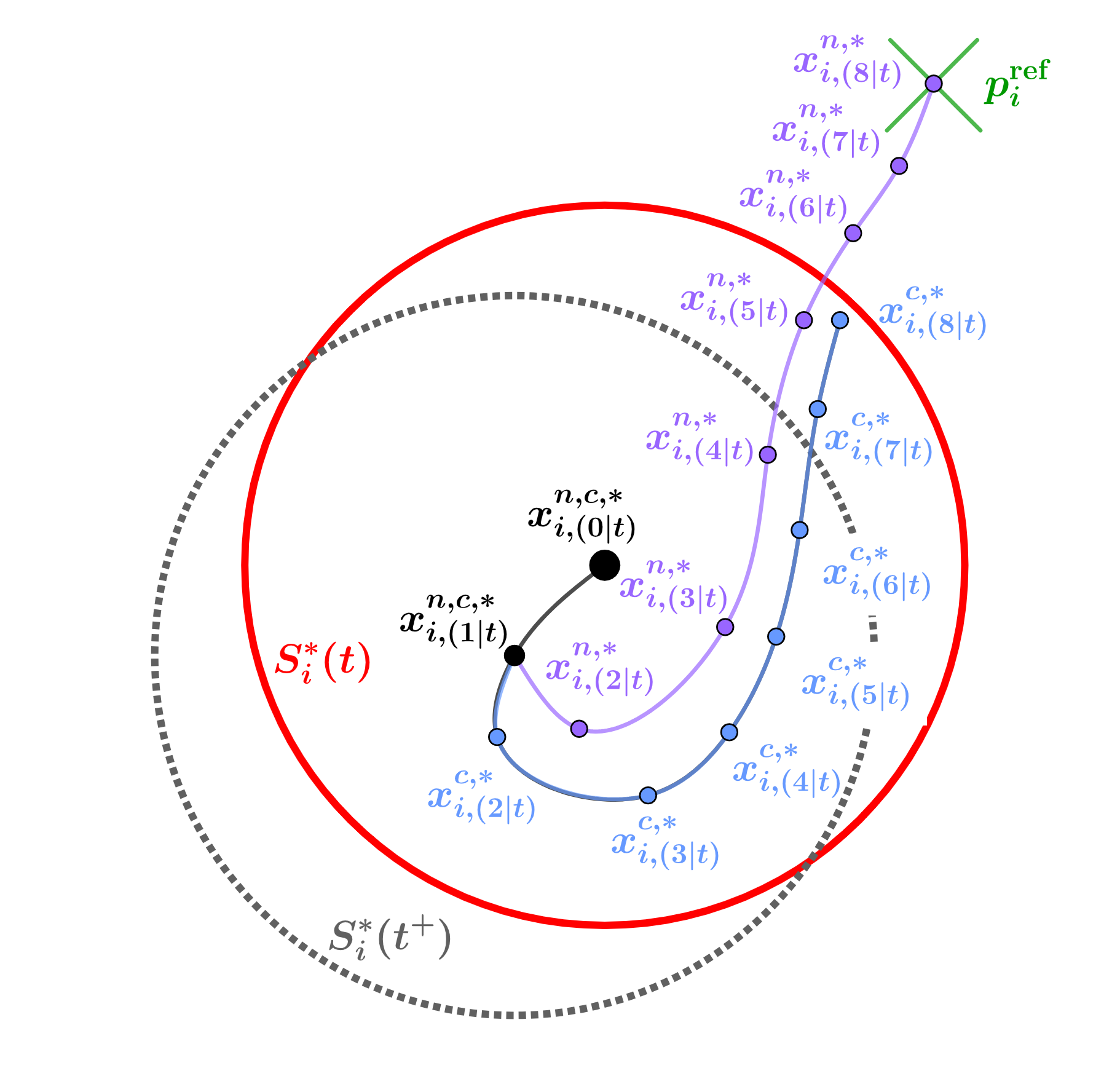}
    \caption{Loss of recursive feasibility without constraint \eqref{eq:tail_containment_constraint}. Although the local MPC problem is feasible at time $t$ with active safe set $S^\ast(t)$ (red), the naive update to $t^+$ yields a set $S^\ast(t^+)$ (dashed) that no longer contains the complete tail of the contingency plan (blue) from previous time step. This violates the safe-set containment constraints and renders the next-step problems infeasible.}
    \label{fig:TailContainmentInfeasibility}
\end{figure}
\vspace*{-0.2cm}
\subsubsection{Collision avoidance w.r.t.\ static obstacles}
\label{subsec:CollisionAvoidance}
\vspace*{-0.1cm}

Collision avoidance with respect to static obstacles is enforced explicitly along the
contingency trajectory. Let $\mathcal{O}\subset\mathbb{R}^{n_p}$ denote the obstacle set (union of closed obstacle
regions) and define a continuous obstacle-avoidance function
$h:\mathbb{R}^{n_p}\rightarrow\mathbb{R}$ such that
\[
h(p_i(t))\ge 0 \ \Longleftrightarrow\ \mathcal{B}_i\cap \mathcal{O} = \emptyset.
\]
Then the following hard constraints are imposed for all predicted contingency positions:
\begin{equation}
h(p^{\mathrm{c}}_{i,(k|t)}) \ge 0,
\qquad \forall\, i\in\mathcal{I},\ \forall\, k\in\mathbb{Z}_0^{N_{\mathrm{c}}}.
\label{eq:static_obstacle_constraints}
\end{equation}

\vspace*{-0.2cm}
\subsubsection{Lyapunov-type decrease constraint}
\label{subsec:dcmpc_convergence_constraint}
\vspace*{-0.1cm}

Following the contingency MPC rationale, a decrease in a contingency-cost functional is enforced through an inequality constraint.
Let $\ell_i^{\mathrm{c}}(\cdot,\cdot)$ denote a continuous, nonnegative contingency stage cost, and define
\begin{equation}
\begin{aligned}
J_i^{\mathrm{c}}(t)
:=
\sum_{k=0}^{N_{\mathrm{c}}-1}
&\ell_i^{\mathrm{c}}\!\big(x^{\mathrm{c}}_{i,(k|t)}-\bar{x}^{\mathrm{c}}_i(t),\,u^{\mathrm{c}}_{i,(k|t)}-\bar{u}^{\mathrm{c}}_i(t)\big) \\
&\;+\;
V_i^{\mathrm{c}}\!\big(\bar{x}^{\mathrm{c}}_i(t), x_i^{\mathrm{ref}}\big).
\end{aligned}
\label{eq:c_cost_def}
\end{equation}
A bound $\hat{J}_i^{\mathrm{c}}(t)\in\mathbb{R}_+$ is maintained recursively and the following constraint is imposed:
\begin{equation}
J_i^{\mathrm{c}}(t) \le \hat{J}_i^{\mathrm{c}}(t).
\label{eq:c_cost_decrease_constraint}
\end{equation}
The update law for $\hat{J}_i^{\mathrm{c}}(t)$ is constructed from the shifted tail of the previously optimal contingency trajectory and yields the one-step bound
\begin{equation}
\hat J_i^{\mathrm c}(t^+)
:=
J_i^{\mathrm c,\ast}(t)
-
\ell_i^{\mathrm{c}}\!\big(x_i(t)-\bar{x}_i^{\mathrm{c},\ast}(t),\,u_i(t)-\bar{u}_i^{\mathrm{c},\ast}(t)\big),
\label{eq:lyap_bound_update_main}
\end{equation}
where $x_i(t) = x_{i,(0|t)}^{c,\ast}$ and $u_i(t)=u^{\mathrm c,\ast}_{i,(0|t)}$ is the applied shared first input. The feasibility of this shifted-tail bound and its role in the convergence proof are established in the Appendix.

\begin{remark}[Equilibrium]
\label{rem:tv_ss_case_split}
Since the optimizer may select a different terminal contingency equilibrium state $\bar{x}^c_i(t)$ at different time steps, two situations can occur.

\begin{enumerate}
\item[(i)] \emph{Reference becomes reachable within the active safe set.}
There exists a time instant $\bar{t}\in\mathbb{Z}_+$ such that the global reference is feasible as a
contingency terminal state thereafter, i.e.,
\[
\exists\,\bar{t}\in\mathbb{Z}_+:\quad
\bar{x}^{c}_i(t)=x_i^{\mathrm{ref}} \qquad \forall t\ge \bar{t},
\]
In this case the selected terminal equilibrium  is constant for all $t\ge\bar{t}$ and the Lyapunov-type analysis
yields convergence to the reference state.

\item[(ii)] \emph{Reference is not reachable due to persistent blocking.}
There exists a time instant $\bar{t}\in\mathbb{Z}_+$ such that the reference cannot be selected as a feasible terminal state, and the optimal reachable terminal state remains constant thereafter, i.e.,
\[
\exists\,\bar{t}\in\mathbb{Z}_+,\ \exists\,\bar{x}^{c,\star}_i:\quad
\bar{x}^{c}_i(t)=\bar{x}^{c,\star}_i \qquad \forall t\ge \bar{t}.
\]
Then the Lyapunov-type decrease constraint implies convergence to the fixed terminal state $\bar{x}^{c,\star}_i$.
\end{enumerate}

\vspace*{-0.2cm}
If $\bar{x}^c_i(t)$ changes infinitely often, the constraint \eqref{eq:c_cost_decrease_constraint}
still enforces a monotonic decrease of the contingency-cost bound generated by the shifted-tail update,
but a standard ``convergence to a fixed equilibrium'' statement is no longer immediate. In this case,
the constraint is interpreted as enforcing progress towards the \emph{currently selected} contingency
terminal state and providing a stabilizing regularization of the contingency plan. In general, the above properties
guarantee neither convergence to the global reference
nor avoidance of deadlocks.
\end{remark}

\vspace*{-0.1cm}
\subsection{Local finite-horizon optimal control problem}
\label{subsec:dcmpc_problem_statement}
\vspace*{-0.1cm}

For each agent $i\in\mathcal{I}$ and time $t\in\mathbb{Z}_+$, a local finite-horizon optimal control problem (FHOCP) is solved:\vspace*{-0.2cm}
\refstepcounter{equation}\label{prob:local_fhocp}\addtocounter{equation}{-1}
\begin{align}
& \min_{X_i^{\mathrm{n,c}},\,U_i^{\mathrm{n,c}}\,\bar{x}_i^{\mathrm{c}},\,\bar{u}_i^{\mathrm{c}}}\
  J_i\big(X_i^{\mathrm{n}}(t), U_i^{\mathrm{n}}(t), \bar{x}^{\mathrm{c}}_i(t), x_i^{\mathrm{ref}}\big) \tag{15a}
  \label{prob:local_fhocp_obj}
\\
& \text{subject to:} \notag
\\
& x^{\mathrm{n}}_{i,(0|t)} = x^{\mathrm{c}}_{i,(0|t)} = x_i(t) \tag{15b}
  \label{prob:local_fhocp_init}
\\
& u^{\mathrm{n}}_{i,(0|t)} = u^{\mathrm{c}}_{i,(0|t)} \tag{15c}
  \label{prob:local_fhocp_shared}
\\
& x^{\mathrm{n}}_{i,(k+1|t)} = f_i \big(x^{\mathrm{n}}_{i,(k|t)}, u^{\mathrm{n}}_{i,(k|t)}\big),
  \forall k\in\mathbb{Z}_0^{N_{\mathrm{n}}-1} \tag{15d}
  \label{prob:local_fhocp_dyn_n}
\\
& x^{\mathrm{c}}_{i,(k+1|t)} = f_i \big(x^{\mathrm{c}}_{i,(k|t)}, u^{\mathrm{c}}_{i,(k|t)}\big),
  \forall k\in\mathbb{Z}_0^{N_{\mathrm{c}}-1} \tag{15e}
  \label{prob:local_fhocp_dyn_c}
\\
& x^{\mathrm{n}}_{i,(k+1|t)}\in\mathcal{X}_i,\ \ u^{\mathrm{n}}_{i,(k|t)}\in\mathcal{U}_i,
  \forall k\in\mathbb{Z}_0^{N_{\mathrm{n}}-1} \tag{15f}
  \label{prob:local_fhocp_bounds_n}
\\
& x^{\mathrm{c}}_{i,(k+1|t)}\in\mathcal{X}_i,\ \ u^{\mathrm{c}}_{i,(k|t)}\in\mathcal{U}_i,
  \forall k\in\mathbb{Z}_0^{N_{\mathrm{c}}-1} \tag{15g}
  \label{prob:local_fhocp_bounds_c}
\\
& h \bigl(p^{\mathrm{c}}_{i,(k|t)}\bigr) \ge 0,
  \forall k\in\mathbb{Z}_0^{N_{\mathrm{c}}} \tag{15h}
  \label{prob:static_obstacle_constraintsMPC}
\\
& x^{\mathrm{c}}_{i,(N_{\mathrm{c}}|t)} = \bar{x}^{\mathrm{c}}_i(t) \tag{15i}
  \label{prob:local_fhocp_terminal_state}
\\
& \bar{x}^{\mathrm{c}}_i(t)
= f_i \bigl(\bar{x}^{\mathrm{c}}_i(t),\bar{u}^{\mathrm{c}}_i(t)\bigr),
\ \bigl(\bar{x}^{\mathrm{c}}_i(t), 
\bar{u}^{\mathrm{c}}_i(t)\bigr)\in \mathcal{X}_i \times \mathcal{U}_i
\tag{15j}
\label{prob:local_fhocp_terminal_equilibrium}
\\
& C_i\bar{x}^{\mathrm{c}}_i(t)\in S_i^\ast(t) \tag{15k}
  \label{prob:local_fhocp_terminal_position}
\\
& \|p^{\mathrm{c}}_{i,(k|t)} - c_i(t)\| \le R_i(t) - r_i,
  \forall k\in\mathbb{Z}_0^{N_{\mathrm{c}}} \tag{15l}
  \label{prob:local_fhocp_containment}
\\
& \|p^{\mathrm{c}}_{i,(l|t)} - c_i(x^{\mathrm{c}}_{i,(k|t)})\| \leq R_i(x^{\mathrm{c}}_{i,(k|t)}) - r_i, \tag{15m}
  \label{prob:local_fhocp_tail}
\\
& \qquad \qquad \qquad \forall k\in\mathbb{Z}_0^{N_{\mathrm{c}}},\ \forall l\in\mathbb{Z}_k^{N_{\mathrm{c}}} \notag
\\
& J_i^{\mathrm{c}}(t) \le \hat{J}_i^{\mathrm{c}}(t) \tag{15n}
  \label{prob:local_fhocp_lyap}
\end{align}
\setcounter{equation}{15}

Problem~\eqref{prob:local_fhocp} is, generally, a non-convex Nonlinear Program (NLP), even if the dynamics are linear. It is solved in receding-horizon fashion. The applied input is the shared first input
\eqref{prob:local_fhocp_shared}.

\begin{remark}[Performance and safety]
Inter-agent collision-avoidance constraints are not explictly imposed on
\(X_i^\mathrm n\). Instead, safety is ensured by \(X_i^\mathrm c\) via
\eqref{prob:local_fhocp_containment}--\eqref{prob:local_fhocp_tail}. The nominal part can be used flexibly for performance shaping, e.g., by additional objectives or soft constraints.
\end{remark}

\vspace*{-0.1cm}
\section{\textit{FoS} update of the local safe sets}
\label{sec:FreezeandshiftOperator}
\vspace*{-0.1cm}

Recall that each agent $i$ is associated with an active safe set as described in~\eqref{eq:local_safe_set}.
After solving the local FHOCPs at time \(t\) and applying the
shared first input, the closed-loop dynamics of each agent follow
\eqref{eq:agent_dynamics}.
The successor state of agent \(i\) is denoted by
\(
x_i(t^+) \in \mathcal X_i.
\)
Recall the deterministic state-dependent safe-set generator
\(\Gamma_i\) introduced in Section~\ref{sec:local_safe_sets}.
Given the successor state \(x_i(t^+)\), the candidate safe set for the
next time step is defined as
\begin{equation}
\tilde S_i(t^+)
:=
\Gamma_i\bigl(x_i(t^+)\bigr)
=
\mathbb{B}\bigl(\tilde c_i(t^+), \tilde R_i(t^+)\bigr),
\label{eq:candidate_safe_set}
\end{equation}
where \(\tilde c_i(t^+):=c_i(x_i(t^+))\) and \(\tilde R_i(t^+):=R_i(x_i(t^+))\). After observing the successor states, agent $i$ evaluates three geometric conditions as part of the \emph{FoS} rule: (i) whether its candidate $\tilde S_i(t^+)$ intersects the candidate $\tilde S_j(t^+)$ of another agent $j\neq i$; (ii) whether its candidate intersects the currently active safe set $S_j^\ast(t)$ of another agent $j\neq i$; and (iii) whether all candidate safe sets are pairwise disjoint. Condition (iii) enables a common shift. otherwise, conditions (i) and (ii) determine the individual updates.
The common shift indicator is defined as
\begin{equation}
\delta(t):=
\begin{cases}
1, & \tilde S_j(t^+)\cap\tilde S_p(t^+)=\emptyset
      \quad\forall j,p\in\mathcal I,\ j\neq p,\\[0.2em]
0, & \text{otherwise}.
\end{cases}
\label{eq:common_shift_indicator}
\end{equation}
The individual overlap indicator $\chi_i(t)$ is defined as
\begin{equation}
\label{eq:freezeIndicator}
\chi_i(t)
:=
\begin{cases}
1 \ \text{if} & \exists j\in\mathcal{I}\setminus\{i\}: 
\begin{aligned}[t]
&\tilde{S}_i(t^+)\cap \tilde{S}_j(t^+)\neq\emptyset \\
\ \ &\text{or}\ \tilde{S}_i(t^+)\cap S_j^\ast(t)\neq\emptyset,
\end{aligned}
\\[0.2em]
0 & \text{otherwise}.
\end{cases}
\end{equation}
Since all local safe sets are closed Euclidean balls, the overlap condition is equivalent to the distance tests
\begin{align}
\exists j\neq i:\quad
&\|\tilde c_i(t^+)-\tilde c_j(t^+)\|
 \le \tilde R_i(t^+)+\tilde R_j(t^+)\nonumber\\
&\text{or }\|\tilde c_i(t^+)-c_j(t)\|
 \le \tilde R_i(t^+)+R_j(t).
\label{eq:freeze_distance_test}
\end{align}
The \emph{FoS} update rule combines the common shift indicator and the local overlap indicator:
\begin{equation}
\label{eq:freeze_shift_update}
\begin{aligned}
(c_i&(t^+),R_i(t^+)) :=\\
&\begin{cases}
(\tilde c_i(t^+),\tilde R_i(t^+)), &\text{if } \delta(t)=1\ \text{or}\ \chi_i(t)=0, \\
(c_i(t),R_i(t)),
&\text{if } \delta(t)=0 \ \text{and} \ \chi_i(t)=1.
\end{cases}
\end{aligned}
\end{equation}
Equivalently,
\[
S_i^\ast(t^+)=
\begin{cases}
\tilde S_i(t^+), \ &\text{if } \delta(t)=1\ \text{or}\ \chi_i(t)=0, \\[0.2em]
S_i^\ast(t), \ &\text{if } \delta(t)=0 \ \text{and} \ \chi_i(t)=1.
\end{cases}
\]By Assumption \ref{ass:information_pattern} and the deterministic construction of the candidate safe sets, each agent $i$ can reconstruct $S_j^\ast(t)$ and $\tilde S_j(t^+)$ for all $j\neq i$ from the current state and one-step memory of the previously active safe sets.
\begin{remark}[Safe-set geometry and updates]
The geometric separation and shifted-plan arguments are
not intrinsically restricted to balls. Other representations,
such as polytopes or buffered Voronoi regions, may be used
provided that the corresponding generators, containment
constraints, and FoS updates satisfy the required
recoverability, reconstructability, and separation conditions.
Disjointness of newly generated regions alone does
not preserve contingency-plan feasibility: recomputed
Voronoi boundaries may exclude a previously feasible
shifted tail. The FoS rule and tail-containment
constraints jointly preserve separation and
shifted-plan feasibility.
\end{remark}

Independent updates
$S_i^\ast(t^+)=\Gamma_i(x_i(t^+))$
do not necessarily preserve pairwise disjointness; see Fig.~\ref{fig:FreezeShiftRule3Agents}. Such overlaps can invalidate
the geometric separation used in the contingency constraints and may
lead to loss of recursive feasibility.
The \textit{FoS} rule prevents this by decoupling the model-dependent
generation of candidate safe sets from the geometric decision whether
to \textit{shift} the active safe set to the new candidate or to
\textit{freeze} it at its previously active value.

The closed-loop algorithm of the proposed decentralized contingency MPC scheme is given below (Algorithm 1).
\begin{algorithm}[H]
\caption{Decentralized Contingency MPC with \textit{FoS} Safe Sets}
\label{alg:dcmpc_freeze_shift}
\begin{algorithmic}[1]
\Require Agent set $\mathcal I$, sampling time \(T_s\), horizons \(N_{\mathrm n},N_{\mathrm c}\), local constraints, $x_i^{\mathrm{ref}}$, and generators \(\Gamma_i(\cdot)\)
\State Initialize \(t\gets0\) and pairwise disjoint safe sets \(S_i^\ast(0)\), \(i\in\mathcal I\)
\While{termination criterion not met}
    \ForAll{\(i\in\mathcal I\) \textbf{(in parallel)}}
        \State Measure \(x_i(t)\), observe \(\{x_j(t)\}_{j\neq i}\), and reconstruct \(\{S_j^\ast(t)\}_{j\in\mathcal I}\)
        \State Solve the local FHOCP~\eqref{prob:local_fhocp}
        \State \(u_i(t)\gets u^{\mathrm n,\ast}_{i,(0|t)}=u^{\mathrm c,\ast}_{i,(0|t)}\)
    \EndFor
    \State Apply \(\{u_i(t)\}_{i\in\mathcal I}\) and measure \(\{x_i(t^+)\}_{i\in\mathcal I}\)
    \State Compute all candidates \(\tilde S_i(t^+)\) via~\eqref{eq:candidate_safe_set} and the common indicator \(\delta(t)\) via~\eqref{eq:common_shift_indicator}
    \ForAll{\(i\in\mathcal I\)}
        \State Update \(S_i^\ast(t^+)\) via~\eqref{eq:freeze_shift_update} and \(\hat J_i^{\mathrm c}(t^+)\) via~\eqref{eq:lyap_bound_update_main}
    \EndFor
    \State \(t\gets t^+\)
\EndWhile
\end{algorithmic}
\end{algorithm}

\begin{figure}[h]
    \centering
    \includegraphics[width=0.45\textwidth]{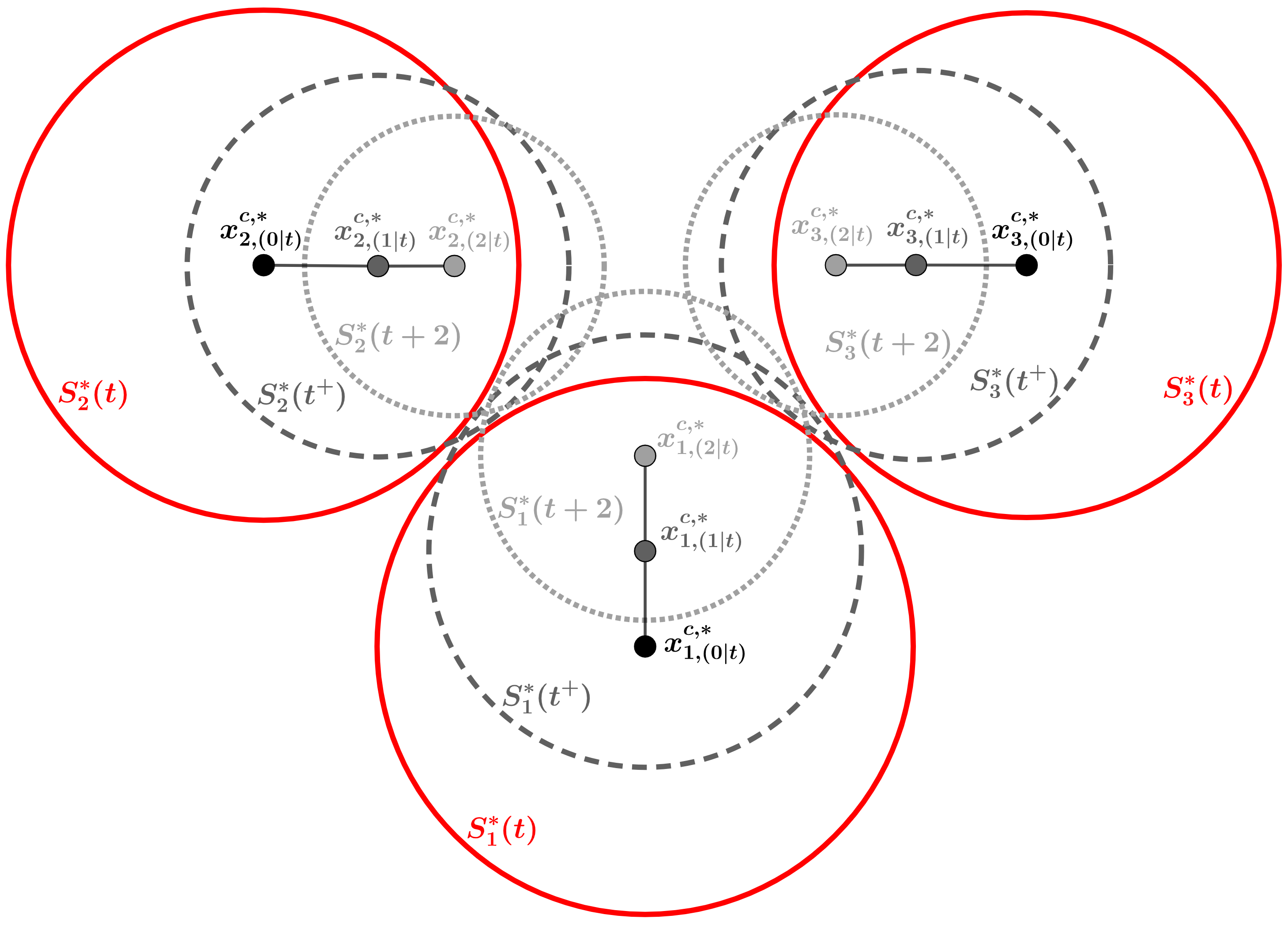}
    \caption{Loss of safe-set separation under naive
    updates. Initially disjoint active safe sets at
    $t$ (red) produce overlapping candidates at $t^+$
    (dashed), invalidating the set-separation.}
    \label{fig:FreezeShiftRule3Agents}
\end{figure}

\vspace*{-0.1cm}
\subsection{Plug-and-play (\textit{PnP}) operation}
\label{subsec:plug_and_play}
\vspace*{-0.1cm}

\textit{PnP} operation allows agents to enter or leave the
workspace online while preserving safety and recursive
feasibility (Algorithm 2). A joining agent has no access to the previously
active safe sets. Admission is therefore restricted to sampling
instants immediately following a common shift, i.e.,
$\delta(t_{\mathrm{join}}^-)=1$, evaluated for the existing
agent set $\mathcal I$. By the \textit{FoS} update rule,
\begin{equation}
S_j^\ast(t_{\mathrm{join}})
=\Gamma_j(x_j(t_{\mathrm{join}})),
\qquad \forall j\in\mathcal I.
\label{eq:pp_exact_reconstruction}
\end{equation}
Thus, the joining agent reconstructs all existing active safe
sets from the current states and the known generators,
without requiring their previous history.

\vspace*{-0.2cm}
\begin{assumption}[Feasible join initialization]
\label{ass:pp_join_feasible}
At a sampling instant $t_{\mathrm{join}}$ satisfying
$\delta(t_{\mathrm{join}}^-)=1$, a joining agent
$i\notin\mathcal I$ initializes
$S_i^\ast(t_{\mathrm{join}})
=\Gamma_i(x_i(t_{\mathrm{join}}))$ such that
\begin{equation}
S_i^\ast(t_{\mathrm{join}})
\cap\Gamma_j(x_j(t_{\mathrm{join}}))
=\emptyset,
\qquad \forall j\in\mathcal I.
\label{eq:pp_join_nonoverlap}
\end{equation}
The local FHOCP~\eqref{prob:local_fhocp} is feasible
for the joining agent at $t_{\mathrm{join}}$.
\end{assumption}

\vspace*{-0.2cm}
Successful joins and leaves are recognized consistently
by all agents at a common sampling instant. The active
agent set is updated before the local FHOCPs are solved
and the subsequent safe-set updates are evaluated.
\begin{algorithm}[H]
\caption{\emph{PnP} join protocol}
\label{alg:join_protocol}
\begin{algorithmic}[1]
\Require Safe-set generators $\Gamma_j(\cdot)$ and local
FHOCP ingredients
\State Observe $\{x_j(t_{\mathrm{join}})\}_{j\in\mathcal I}$
and construct their generated safe sets
\If{$\delta(t_{\mathrm{join}}^-) = 0$}
    \State Reject or postpone the join
    \State \textbf{return}
\EndIf
\State Reconstruct
$S_j^\ast(t_{\mathrm{join}})
=\Gamma_j(x_j(t_{\mathrm{join}}))$ for all $j\in\mathcal I$
\State Initialize
$S_i^\ast(t_{\mathrm{join}})
=\Gamma_i(x_i(t_{\mathrm{join}}))$
\If{\( \exists j\in\mathcal I, \ \text{s.t.} \ 
S_i^\ast(t_{\mathrm{join}})
\cap\Gamma_j(x_j(t_{\mathrm{join}}))
\neq\emptyset.
\)}
    \State Reject or postpone the join
    \State \textbf{return}
\EndIf
\State Update $\mathcal I\gets\mathcal I\cup\{i\}$
\State Continue Algorithm~\ref{alg:dcmpc_freeze_shift}
\end{algorithmic}
\end{algorithm}
\vspace*{-0.2cm}
A leaving protocol is immediate in the decentralized setting: When an agent exits, it is removed from the set of observed agents and from the locally reconstructed constraints. The remaining agents continue to solve their local contingency MPC problems, and the closed-loop guarantees are preserved by construction.

\vspace*{-0.1cm}
\section{System-theoretic analysis}
\label{sec:analysis}
\vspace*{-0.1cm}

This section establishes the main closed-loop guarantees of the proposed decentralized contingency MPC scheme combined with the \textit{FoS} update rule.
First, for a fixed set of active agents, recursive feasibility and collision avoidance are addressed.
After that, the Lyapunov-type decrease induced by the shifted-tail bound update is analyzed and the corresponding convergence statement is derived.
The \emph{PnP} protocol introduced in Section~5 is treated separately and does not enter the core recursive-feasibility statement below.
The recursive-feasibility and Lyapunov-type arguments developed in this section follow the classical shifted-candidate logic of MPC and NMPC, adapted here to the decentralized contingency formulation with \textit{FoS}-based safe-set updates~\cite{ChenAllgoewer1998,Mayne2000}.
The detailed proofs are deferred to the Appendix.
\vspace*{-0.2cm}
\begin{assumption}[Initial feasibility and separation]
\label{ass:initial_feasibility}
At time $t_0=0$, the local FHOCP~\eqref{prob:local_fhocp} is feasible for every agent $i\in\mathcal I$, and the initial active safe sets are pairwise disjoint, i.e.,
\[
S_i^\ast(0)\cap S_j^\ast(0)=\emptyset,
\qquad \forall i\neq j.
\]
\end{assumption}

\vspace*{-0.2cm}
\begin{assumption}[Nominal completion]
\label{ass:nominal_completion}
Whenever a feasible contingency candidate at time $t^+$ is available, there exists a feasible nominal candidate at time $t^+$ satisfying the nominal initialization, dynamics, and state/input constraints in~\eqref{prob:local_fhocp}.
This is, for example, satisfied if the nominal part is chosen identical to the contingency part.
\end{assumption}

\vspace*{-0.2cm}
\begin{definition}[$\mathcal K_\infty$ function]
\label{def:Kinf}
A continuous function $\alpha:\mathbb{R}_+\to\mathbb{R}_+$ is said to belong to class $\mathcal K_\infty$ if
\(
\alpha(0)=0,
\)
$\alpha$ is strictly increasing, and
\(
\alpha(s)\to\infty
\ \text{as } s\to\infty.
\)
\end{definition}

\vspace*{-0.2cm}
\begin{assumption}[Regularity and coercivity]
\label{ass:regularity_cost}
Let
\(
\mathcal Z_i := \mathcal X_i\times\mathcal U_i,
\)
and let $\overline{\mathcal Z}_i$ denote its closure.
For each agent $i\in\mathcal I$, the following properties hold:

\begin{enumerate}
    \item \textbf{Continuity of dynamics and cost.}
    The mappings
    \(
    f_i:\mathbb R^{n_i}\times\mathbb R^{m_i}\to\mathbb R^{n_i},
    \
    \ell_i^{\mathrm c}:\mathbb R^{n_i}\times\mathbb R^{m_i}\to\mathbb R_+,
    \)
    and the offset cost
    \(
    V_i^{\mathrm c}:\mathbb R^{n_i}\times\mathbb R^{n_i}\to\mathbb R_+
    \)
    are continuous on the relevant admissible sets.

    \item \textbf{Incremental continuity of the dynamics.}
    There exists a function $\alpha_{f,i}\in\mathcal K_\infty$ such that
    \begin{align*}
    \|f_i(x,u)-f_i(\tilde x,\tilde u)\|
    \le
    \alpha_{f,i}\!\left(
    \left\|
    \begin{bmatrix}
    x-\tilde x\\
    u-\tilde u
    \end{bmatrix}
    \right\|
    \right),
    \\
    \forall (x,u),(\tilde x,\tilde u)\in\overline{\mathcal Z}_i.
    \end{align*}

    \item \textbf{Positive definiteness.}
    There exist functions
    \(
    \underline{\alpha}_{\ell,i},\,\overline{\alpha}_{\ell,i}\in\mathcal K_\infty
    \)
    such that
    \begin{equation}
    \begin{aligned}
    \underline{\alpha}_{\ell,i}\!\bigl(\|(e,\delta u)\|\bigr)
    \le
    \ell_i^{\mathrm c}(e,\delta u)
    \le
    \overline{\alpha}_{\ell,i}\!\bigl(\|(e,\delta u)\|\bigr),
    \\
    \forall (e,\delta u)\in\overline{\mathcal Z}^{\mathrm e}_i,
    \label{eq:stage_cost_bounds}
    \end{aligned}
    \end{equation}
    where $\overline{\mathcal Z}^{\mathrm e}_i$ denotes the relevant closed admissible set in error coordinates: $e = x_i^c-\bar{x}_i^c$ and $\delta u = u_i^\mathrm{c}-\bar{u}_i^c$.
    In particular, $\ell_i^{\mathrm c}$ is positive definite with respect to $(e,\delta u)=(0,0)$.

    \item \textbf{Positive definiteness of the offset cost.}
    There exist functions
    \(
    \underline{\alpha}_{V,i},\,\overline{\alpha}_{V,i}\in\mathcal K_\infty
    \)
    such that
    \[
    \underline{\alpha}_{V,i}\!\bigl(\|\bar{x}_i^c- x_i^{\mathrm{ref}}\|\bigr)
    \le
    V_i^{\mathrm c}(\bar{x}_i^c,x_i^{\mathrm{ref}})
    \le
    \overline{\alpha}_{V,i}\!\bigl(\|\bar{x}_i^c- x_i^{\mathrm{ref}}\|\bigr)
    \]
    for all relevant admissible equilibrium candidates $\bar x_i^{\mathrm c}$.
\end{enumerate}
\end{assumption}

\begin{lemma}[Disjointness invariance]
\label{lem:disjoint_invariance}
Under Assumption~\ref{ass:initial_feasibility} and the update rule~\eqref{eq:freeze_shift_update}, the active safe sets remain pairwise disjoint for all $t\in\mathbb{Z}_+$, i.e.,
\[
S_i^\ast(t)\cap S_j^\ast(t)=\emptyset,
\qquad \forall i\neq j.
\]
\end{lemma}

\begin{theorem}[Recursive feasibility]
\label{thm:rec_feas}
Suppose that Assumptions~\ref{ass:exact_dynamics}--\ref{ass:contingency_recoverability}, \ref{ass:initial_feasibility},~\ref{ass:nominal_completion}, and~\ref{ass:regularity_cost}(3) hold.
Consider the proposed decentralized contingency MPC scheme with the \textit{FoS} update rule~\eqref{eq:freeze_shift_update} and the shifted-tail update of the Lyapunov bound~\eqref{eq:lyap_bound_update_main}.
Then the local FHOCP~\eqref{prob:local_fhocp} remains feasible for all times, i.e., for every agent $i\in\mathcal I$, \eqref{prob:local_fhocp} is feasible at every  $t\in\mathbb Z_+$.
\end{theorem}

\begin{theorem}[Closed-loop collision avoidance]
\label{thm:safety}
Suppose that the assumptions of Theorem~\ref{thm:rec_feas} hold.
Then, under the proposed decentralized contingency MPC scheme with the \textit{FoS} update rule, the closed-loop execution is collision-free for all times, i.e.,
\[
\mathcal B_i(t)\cap \mathcal B_j(t)=\emptyset,
\qquad \forall i\neq j,\ \forall t\in\mathbb Z_+.
\]
\end{theorem}

\begin{corollary}[Preservation under \emph{PnP} events]
\label{cor:pnp_preservation}
Consider the proposed decentralized contingency MPC scheme with the FoS update rule.

\emph{Join:}
Let a new agent $i_{\mathrm{new}}\notin\mathcal I$ request insertion at time $t_{\mathrm{join}}$.
If Assumption~\ref{ass:pp_join_feasible} holds, then recursive feasibility and collision avoidance are preserved for the augmented active set
\[
\mathcal I^+ := \mathcal I \cup \{i_{\mathrm{new}}\} \quad \forall t\ge t_{\mathrm{join}}.
\]
\emph{Leave:}
If agent $j\in\mathcal I$ leaves the workspace at time $t_{\mathrm{leave}}$, it is removed from the locally reconstructed constraints of the remaining agents. Recursive feasibility and collision avoidance are preserved for the reduced active set
\[
\mathcal I^- := \mathcal I \setminus \{j\} \quad \forall t \geq t_{\mathrm{leave}}.
\]
\end{corollary}

\vspace*{-0.4cm}
The convergence mechanism is induced by the Lyapunov-type constraint~\eqref{prob:local_fhocp_lyap}.
The argument follows the standard receding-horizon value-function logic: feasibility of the shifted contingency candidate yields an explicit feasible upper bound for the successor problem, and the imposed Lyapunov constraint then enforces a monotone decrease of the optimal contingency-cost sequence. This is the same shifted-candidate proof architecture used in classical MPC stability proofs; see, e.g.,~\cite{ChenAllgoewer1998,Mayne2000}.

\vspace*{-0.2cm}
For agent \(i\in\mathcal I\), define the equilibrium-tracking
errors associated with the optimal contingency equilibrium by
\[
e_i(t):=x_i(t)-\bar x_i^{\mathrm c,\ast}(t),
\qquad
\delta u_i(t):=u_i(t)-\bar u_i^{\mathrm c,\ast}(t),
\]
where
\(
u_i(t)=u_{i,(0|t)}^{\mathrm c,\ast}
      =u_{i,(0|t)}^{\mathrm n,\ast}
\)
is the applied shared first input and $x_i(t)$ is as in (\ref{prob:local_fhocp_init}).

\vspace*{-0.1cm}
\begin{proposition}
\label{prop:lyap_decrease}
Fix an agent \(i\in\mathcal I\) and suppose that recursive feasibility of
the local FHOCP~\eqref{prob:local_fhocp} holds for all
\(t\in\mathbb Z_+\). Assume that the Lyapunov
constraint~\eqref{prob:local_fhocp_lyap} is enforced at every time step,
the bound \(\hat J_i^{\mathrm c}(\cdot)\) is updated by the
shifted-tail rule~\eqref{eq:lyap_bound_update_main}, and
Assumption~\ref{ass:regularity_cost} holds. Then
\begin{equation}
J_i^{\mathrm c,\ast}(t^+)
\le
J_i^{\mathrm c,\ast}(t)
-
\ell_i^{\mathrm c}\!\bigl(e_i(t),\delta u_i(t)\bigr),
\qquad \forall t\in\mathbb Z_+.
\label{eq:lyap_decrease_main}
\end{equation}
Hence the sequence $\{J_i^{\mathrm c,\ast}(t)\}_{t\in\mathbb Z_+}$ is monotonically nonincreasing and lower bounded, thus convergent, and
\begin{equation}
\sum_{t=0}^{\infty}
\ell_i^{\mathrm c}\!\bigl(e_i(t),\delta u_i(t)\bigr)
<\infty.
\label{eq:summable_stage_main}
\end{equation}
Consequently,
\(
\ell_i^{\mathrm c}\!\bigl(e_i(t),\delta u_i(t)\bigr)\to 0
\ \text{as } t\to\infty,
\)
and by the lower bound in~\eqref{eq:stage_cost_bounds},
\(
\|(e_i(t),\delta u_i(t))\|\to 0
\ \text{as } t\to\infty.
\)
\end{proposition}

\vspace*{-0.1cm}
\begin{remark}[Interpretation of Proposition~\ref{prop:lyap_decrease}]
\label{rem:prop_lyap_interpretation}
Proposition~\ref{prop:lyap_decrease} is the precise Lyapunov-type statement available without any additional assumption on the time evolution of the selected contingency equilibrium.
It guarantees monotone decrease of the optimal contingency cost, summability of the stage cost, and asymptotic convergence of the closed-loop trajectory in the equilibrium-tracking coordinates $(e_i(t),\delta u_i(t))$.
\end{remark}

\vspace*{-0.1cm}
\begin{corollary}[Convergence to a fixed equilibrium]
\label{cor:fixed_ss_convergence}
Suppose that the assumptions of Proposition~\ref{prop:lyap_decrease} hold and that there exists \(\bar t\in\mathbb Z_+\) such that, for all \(t\ge\bar t\), \(\bar x_i^{\mathrm c,\ast}(t)=\bar x_i^{\mathrm c,\star}\) and \(\bar u_i^{\mathrm c,\ast}(t)=\bar u_i^{\mathrm c,\star}\). Then \(x_i(t)\to \bar x_i^{\mathrm c,\star}\) and \(u_i(t)\to \bar u_i^{\mathrm c,\star}\) as \(t\to\infty\).
\end{corollary}

\vspace*{-0.1cm}
\section{Simulation Study}
\vspace*{-0.1cm}

The simulation results are further illustrated in the supplementary video: \url{https://youtu.be/CQD76o5ZsJw}. \\
They serve as a proof of concept, demonstrating how the proposed framework can be applied to a nonlinear multi-agent system.

\vspace*{-0.1cm}
\subsection{Setup}
\label{subsec:sim_setup}
\vspace*{-0.1cm}

For illustration, the framework of this paper is instantiated for agents that are modeled by planar double-integrator dynamics with quadratic aerodynamic drag, as considered in \cite{ChipadePanagou2021}, in a bounded workspace
\(
\mathcal W \subset \mathbb R^2, \ \text{i.e.,} \ n_p = 2.
\)
For each agent $i\in\mathcal I$, the state is given by
\(
x_i(t)=
[
p_i(t)\ \
v_i(t)
]^\top
\in\mathbb R^4,
\ 
p_i(t)\in\mathbb R^2,
\
v_i(t)\in\mathbb R^2,
\)
and the input is
\(
u_i(t)\in\mathbb R^2.
\)
The continuous-time model is
\[
\dot p_i=v_i,
\qquad
\dot v_i=u_i-c_d\|v_i\|v_i,
\]
where $c_d=0.2\,\mathrm{m}^{-1}$ is a constant drag coefficient.
Evaluating the acceleration at the beginning of each sampling interval gives the discrete-time approximation
\begin{equation}
\begin{aligned}
x_i(t^+)&=f(x_i(t),u_i(t)) \\
&= 
Ax_i(t)+B\bigl(u_i(t)-c_d\|v_i(t)\|v_i(t)\bigr),
\label{eq:sim_double_integrator}
\end{aligned}
\end{equation}
with sampling time $T_s=0.1\,\mathrm{s}$ and system matrices
\[
A=
\begin{bmatrix}
I_2 & T_s I_2\\
0_2 & I_2
\end{bmatrix},
\qquad
B=
\begin{bmatrix}
\frac{1}{2}T_s^2 I_2\\
T_s I_2
\end{bmatrix},
\]
where $I_2\in\mathbb R^{2\times 2}$ denotes the identity matrix and $0_2\in\mathbb R^{2\times 2}$ the zero matrix.
The position output is
\(
p_i(t)=C_i x_i(t),
\
C_i=[ I_2 \ \  0_2 ].
\)
Each agent body is approximated by a disc of radius
\(
r_i=0.2\,\mathrm{m}.
\)
State and input bounds are imposed through the velocity and input limits
\(
\|v_i(t)\|\le 3\,\mathrm{m/s},
\
\|u_i(t)\|\le 3.5\,\mathrm{m/s}^2.
\)
The nominal prediction horizon is chosen as \(N_{\mathrm n}=20\).
The contingency horizon \(N_{\mathrm c}\) is chosen
as short as possible to limit computational effort,
while allowing any admissible shared first input
followed by admissible braking from the maximum
velocity magnitude to a stopped equilibrium,
\(
\bar x_i^{\mathrm c} =
[
\bar p_i^{\mathrm c} \ \ 
0
]^\top,
\ 
\bar u_i^{\mathrm c}=0.
\) Its choice does not affect the geometry of the safe sets
generated here, but an overly short horizon may restrict
admissible maneuvers.
Since the drag assists braking, a sufficient horizon length is
\begin{equation}
N_{\mathrm c}
\ge
1+\left\lceil\frac{v_{\max}}{a_{\max}T_s}\right\rceil .
\label{eq:sim_Nc_choice}
\end{equation}
For the considered model, the safe-set generator uses the
current position as its center. Let
\(s=\|v_i(t)\|\), \(b=|a_{\min}|\), and
\(\hat s^+=\min(v_{\max},s+T_sa_{\max})\).
A simple radius is
\begin{equation}
R_i(t)
=
\frac{T_s}{2}(s+\hat s^+)
+
\frac{(\hat s^+)^2}{2b}
+
\frac{bT_s^2}{8}
+
r_i.
\label{eq:RadiusNonLinModel}
\end{equation}
The first term bounds the displacement during the shared first input, while the remaining terms bound the subsequent braking
distance, conservatively neglecting the
additional braking effect of drag. 
\begin{remark}[Radius selection and body geometry]
For mobile systems admitting a feasible stopping maneuver,
such as suitable ground-robot, vehicle, or multirotor
models, a conservative bound on the displacement along
that maneuver can be used to construct a recoverable
safe set. A tight braking-radius formula is not required;
model-specific refinements may reduce conservatism.
The nonlinear guarantees remain conditional on
Assumption~\ref{ass:contingency_recoverability}.
A smaller radius may restrict admissible speeds and
maneuvers, while recursive feasibility is preserved
under the assumptions of Theorem~\ref{thm:rec_feas}.
Asymmetric robot bodies can be represented conservatively
by an enclosing ball of radius $r_i$ centered at the
modeled position. This ball must contain the entire
body for every admissible orientation.
\end{remark}

\vspace*{-0.2cm}
\subsubsection{MPC Implementation}
\vspace*{-0.1cm}

All local optimal control problems are implemented in \textsc{CasADi} and solved using \textsc{IPOPT}~\cite{IPOPT}.
Solver tolerances are set to $10^{-5}$ for feasibility and optimality, and the maximum number of iterations is capped at $5000$.
Warm-starting is used throughout.
\label{subsec:mpc_impl}

\vspace{-0.1cm}
\noindent\textit{Tracking and equilibrium optimization} \vspace{0.15cm} \\
The quadratic objective in the form of 
\begin{align*}
&J_i\big(X_i^{\mathrm{n}}(t), U_i^{\mathrm{n}}(t), \bar{x}^{\mathrm{c}}_i(t), x_i^{\mathrm{ref}}\big)= \\ 
&\sum_{k=0}^{N_{\mathrm n}-1}
\Big(
\|p^{\mathrm n}_{i,(k|t)}-p_i^{\mathrm{ref}}\|_{Q_p}^2
+\|v^{\mathrm n}_{i,(k|t)}\|_{Q_v}^2
+\|u^{\mathrm n}_{i,(k|t)}\|_{R_u}^2
\Big) \\
& \qquad \ +\|p^{\mathrm n}_{i,(N_{\mathrm n}|t)}-p_i^{\mathrm{ref}}\|_{Q_p}^2 +\gamma\,V_i^{\mathrm{c}}\!\big(\bar{x}^{\mathrm{c}}_i(t), x_i^{\mathrm{ref}}\big),
\end{align*}
with diagonal weights
\(
Q_p = Q_v = R_u=I_2,
\)
and
\[
V_i^{\mathrm{c}}\!\big(\bar{x}^{\mathrm{c}}_i(t), x_i^{\mathrm{ref}}\big) =\|\bar x_i^{\mathrm c}(t)-x_i^{\mathrm{ref}}\|_{P_s}^2,\ \
x^{\mathrm{ref}}_i=\begin{bmatrix}p_i^{\mathrm{ref}}&0\end{bmatrix}^\top,
\]
where $P_s = I_4$ and $\gamma=0.1$ is used in the simulations.

\vspace{-0.1cm}
\noindent\textit{Contingency cost} \vspace{0.15cm} \\
The Lyapunov bound uses the safe-cost functional
\begin{align*}
J_i^{\mathrm c}(t)
=
\sum_{k=0}^{N_{\mathrm c}-1}&
\Big(
\|x^{\mathrm c}_{i,(k|t)}-\bar x_i^{\mathrm c}(t)\|_{Q_s}^2
+\|u^{\mathrm c}_{i,(k|t)}\|_{R_s}^2
\Big) \\
&+\|\bar x_i^{\mathrm c}(t)-x_i^{\mathrm{ref}}\|_{P_s}^2,
\end{align*}
with $Q_s=0.1\cdot I_4$, $R_s= 0.1 \cdot I_2$, and $P_s$ as before.

\vspace{-0.1cm}
\noindent\textit{Nominal soft collision-avoidance constraints} \vspace{0.15cm} \\
For each other agent $j\neq i$ and each nominal stage $k\in\mathbb{Z}_0^{N_{\mathrm n}-1}$ we introduce nonnegative slacks
$s_{ij,k}\ge 0$, and impose the squared-distance constraint
\begin{equation}
\|p^{\mathrm n}_{i,(k|t)}-c_j(t)\|^2 + s_{ij,k}
\;\ge\;
d_{ij,k}^2,
\label{eq:soft_dist_gate}
\end{equation}
where $c_j(t)$ and $R_j(t)$ denote the \emph{currently reconstructed} active safe-set center and radius of agent $j$, and
\[
d_{ij,k}:=
R_i^{\mathrm{gen}}\!\bigl(x^{\mathrm n}_{i,(k|t)}\bigr)
+R_j(t)+d_{\mathrm{buf}},
\quad d_{\mathrm{buf}}=0.1\,\mathrm m.
\]
Here, $R_i^{\mathrm{gen}}(x^{\mathrm n}_{i,(k|t)})$
denotes the radius of the safe-set generator evaluated
at the predicted nominal state at time $t$, using~\eqref{eq:RadiusNonLinModel}. The slack variables are penalized in the nominal objective by
\[
J_i \leftarrow J_i
+
\rho_{\mathrm{nom}}
\sum_{j\in\mathcal I\setminus\{i\}}
\sum_{k=0}^{N_{\mathrm n}-1}
s_{ij,k}, \quad \rho_{\mathrm{nom}}=100.
\]
\vspace*{-0.2cm}
\subsubsection{Increasing agent density}
\vspace*{-0.1cm}
This scenario class evaluates scalability with respect to agent density. The increasing-density study uses \(M=5,10,20\) agents with randomly
sampled initial and reference positions under a minimum separation. Fig.~\ref{fig:inc_density_traj} shows representative collision-free
closed-loop trajectories $M=10 \ \text{and} \ 20$. The normalized contingency
cost in Fig.~\ref{fig:inc_density_cost} decreases as enforced by
\eqref{prob:local_fhocp_lyap}; higher densities mainly increase the
transient phase due to stronger local interactions.
\begin{figure}[h]
    \centering
    \begin{minipage}[t]{0.5\linewidth}
        \centering
        \includegraphics[width=\linewidth]{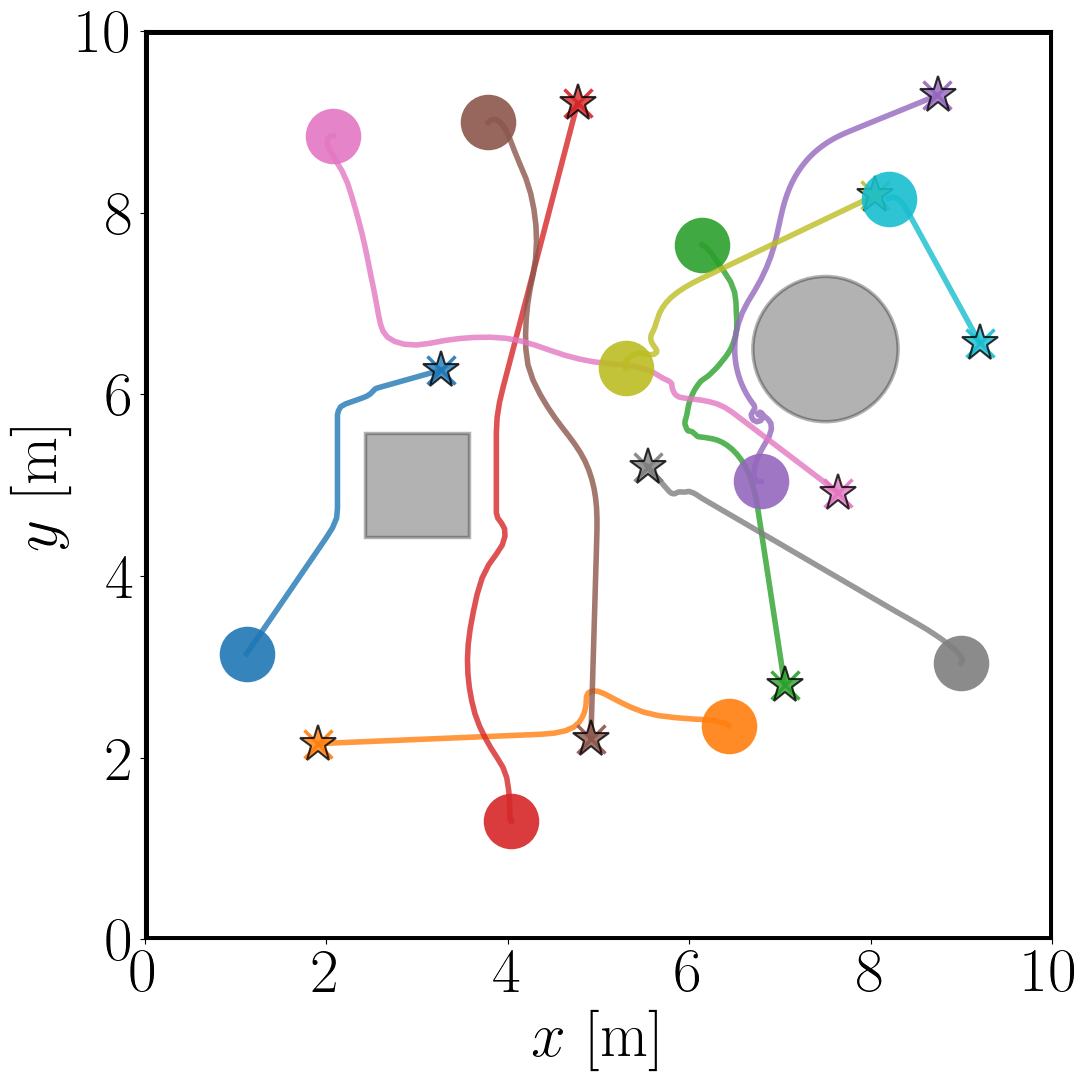}
    \end{minipage}\hfill
    \begin{minipage}[t]{0.5\linewidth}
        \centering
        \includegraphics[width=\linewidth]{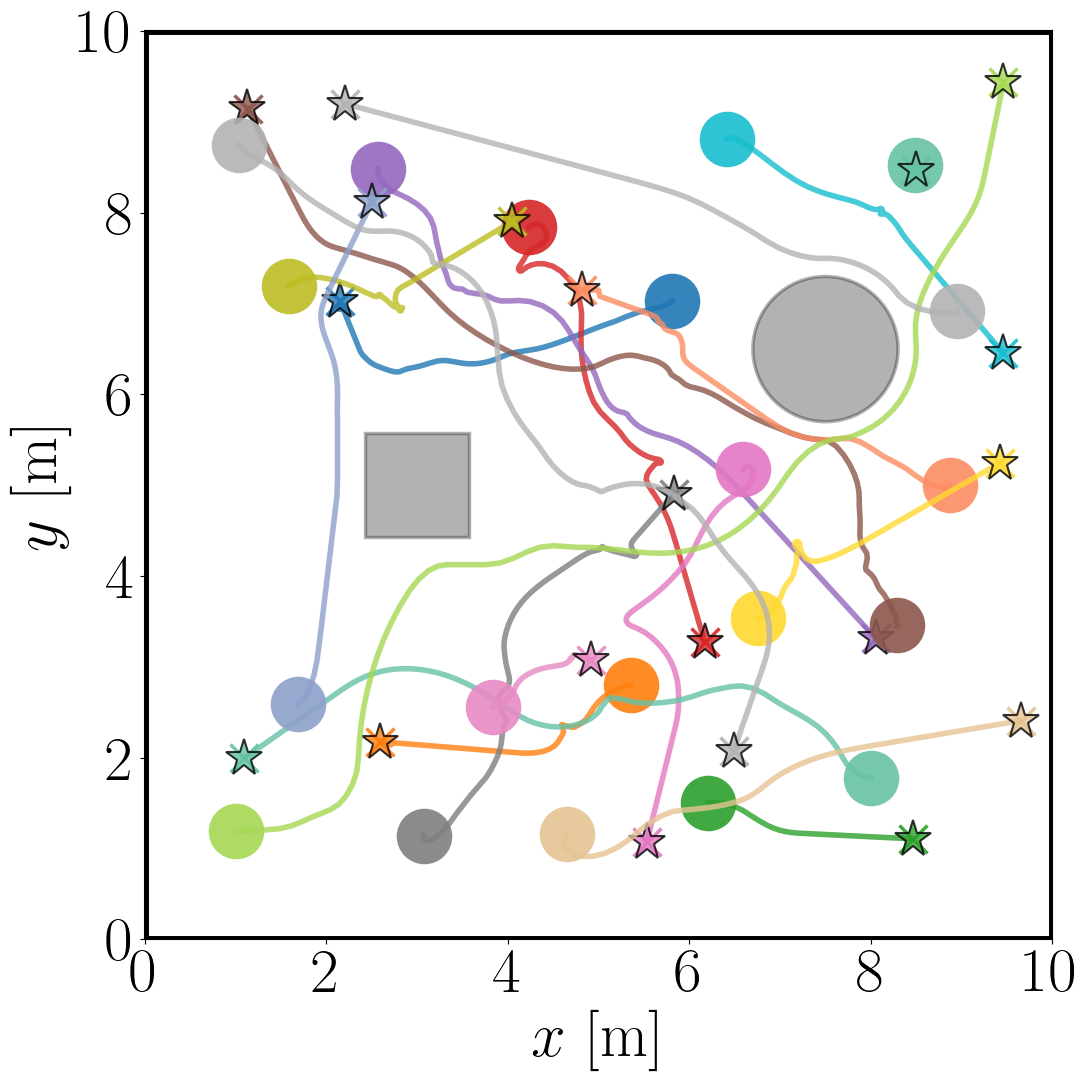}
    \end{minipage}
    \caption{Representative closed-loop trajectories in the increasing-density study: $M=10$ (left), and $M=20$ (right). Filled circles denote initial positions and star markers denote reference positions. The gray objects represent static obstacles. The trajectories remain collision-free in all cases.}
    \label{fig:inc_density_traj}
\end{figure}
\begin{figure}[h!]
        \includegraphics[width=1\linewidth]{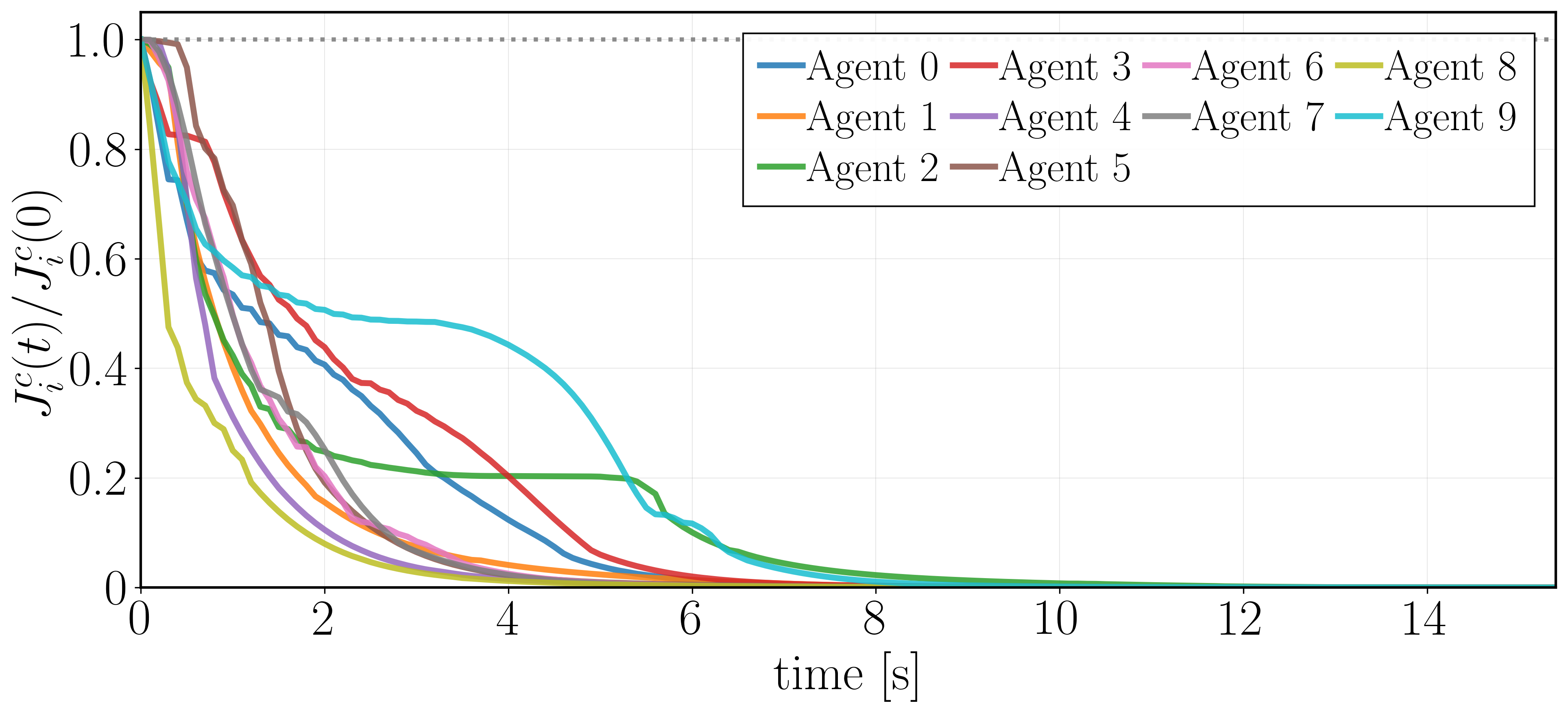}
    \caption{Normalized contingency-cost evolution $J_i^c(t)/J_i^c(0)$ for $M = 10$ agents (see Fig. \ref{fig:inc_density_traj}). The trajectories satisfy the Lyapunov-decreasing behavior induced by the shifted-tail contingency-cost update (constraint \eqref{prob:local_fhocp_lyap}).}
    \label{fig:inc_density_cost}
\end{figure}

Computation times were measured on a workstation equipped
with an AMD Ryzen 9 9950X processor and 64 GB RAM.
The mean wall-clock times per local MPC solve were
$67.38$, $95.06$, and $162.08\,\mathrm{ms}$, with corresponding
standard deviations of $55.73$, $130.57$, and
$320.41\,\mathrm{ms}$ for $M=5,10,20$, respectively.
The implementation is not optimized for computational
efficiency.

\vspace*{-0.2cm}
\subsubsection{Bottleneck}
\vspace*{-0.1cm}

The bottleneck scenario stresses conflict resolution in a narrow passage
formed by two static obstacles. Fig.~\ref{fig:FieldBottleneck} shows
that all agents pass the constriction without collisions. Fig.~\ref{fig:AgentDistanceBottleneck} shows the corresponding
agent and safe-set distances. Only three freeze events occurred
during the simulation, preventing unsafe safe-set updates while
the agents continued towards their targets.
\begin{figure}[h]
    \centering
    \includegraphics[width=0.6\linewidth]{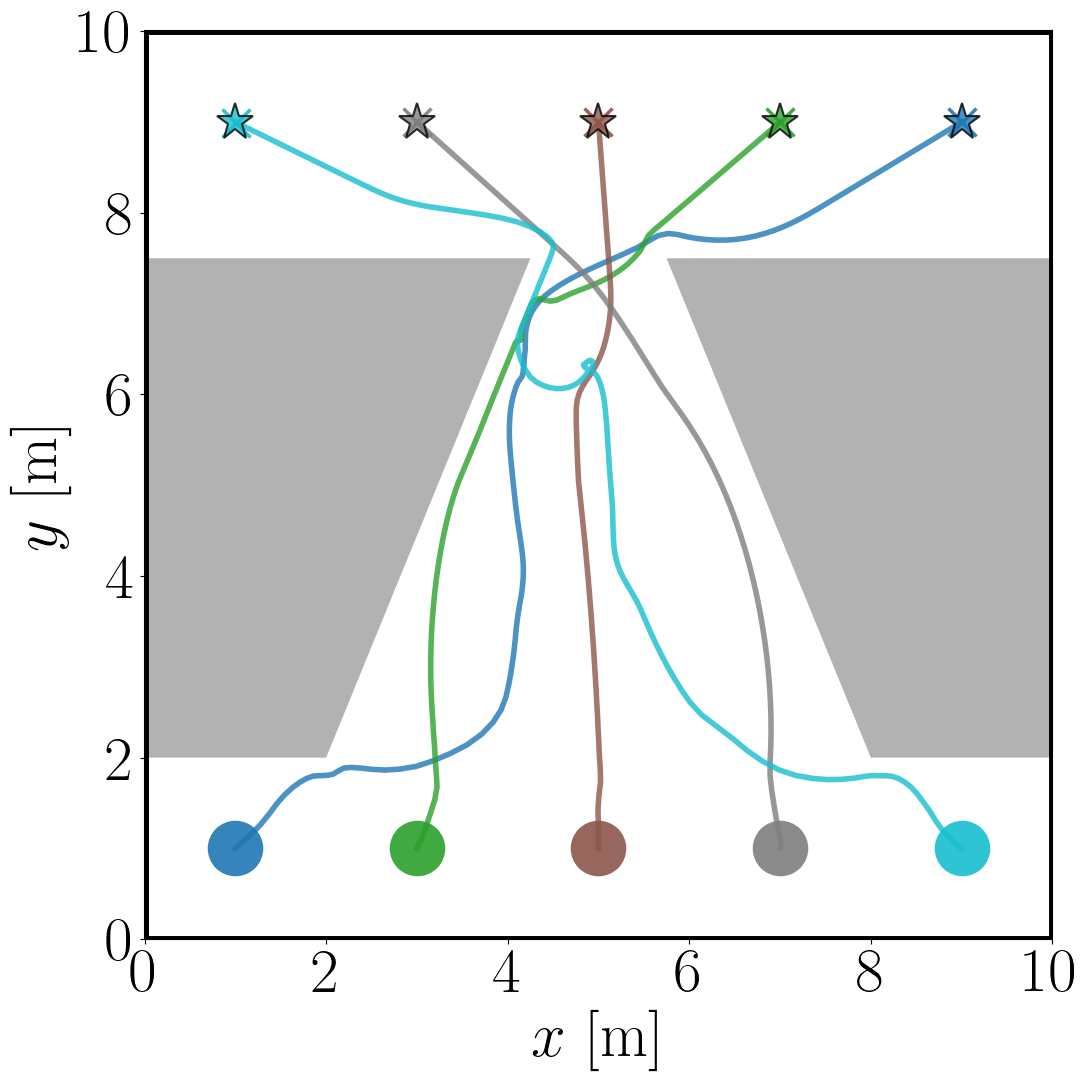}
    \caption{Closed-loop trajectories in the bottleneck scenario. Filled circles denote initial positions and star markers denote target positions. The gray obstacles create a narrow passage.}
    \label{fig:FieldBottleneck}
\end{figure}
\begin{figure}[h]
    \centering
    \includegraphics[width=0.9\linewidth]{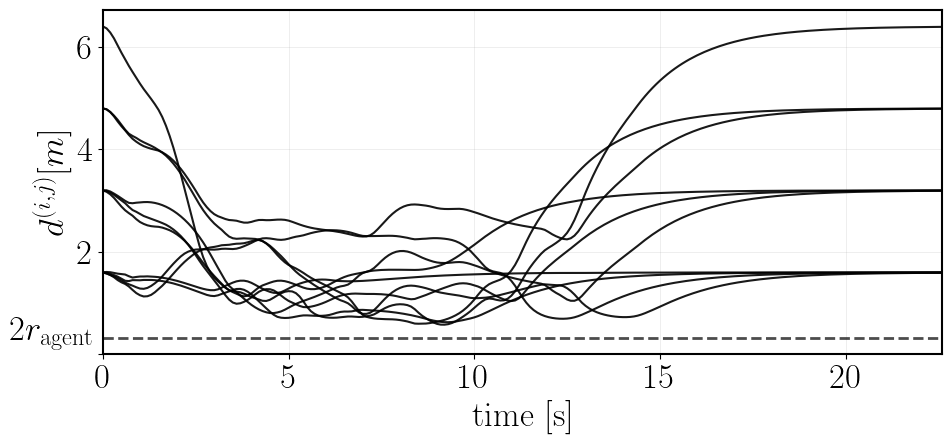}\\[-1.2em]
    \includegraphics[width=0.9\linewidth]{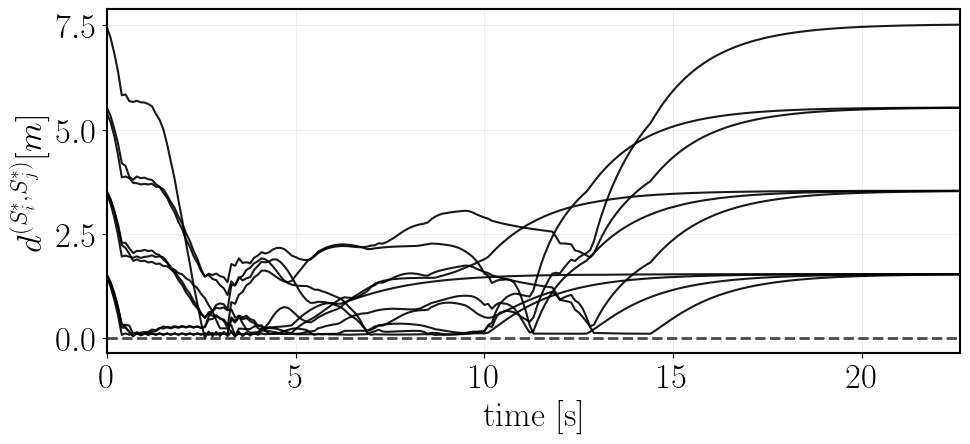}
    \caption{Distance evolution in the \emph{bottleneck} scenario. Top: pairwise distances between agents; the dashed line indicates the minimum admissible distance. Bottom: distances between agent safe sets; values close to zero indicate strong interaction and freeze-operator activation.}
    \label{fig:AgentDistanceBottleneck}
\end{figure}

\vspace*{-0.2cm}
\subsubsection{PnP}
\vspace*{-0.1cm}

The \emph{PnP} experiment evaluates the proposed scheme under online changes in the set of active agents. Starting from a feasible multi-agent configuration, three additional agents join the workspace during closed-loop operation and are initialized according to the join protocol in Section~\ref{subsec:plug_and_play}. Later, a different agent leaves the workspace and is removed from the locally reconstructed constraints of the remaining agents.
Fig.~\ref{fig:PlugAndPlayFigure} summarizes the experiment. The results illustrate that insertion and removal can be handled without trajectory communication and without modifying the local MPC problem structure of the remaining agents. In particular, the deterministic \textit{FoS} update and the (history-free) reconstruction at the join time allow the scheme to preserve collision avoidance and feasibility despite the changing agent configuration.

\begin{figure}[h!]
    \centering
    \includegraphics[width=1\linewidth]{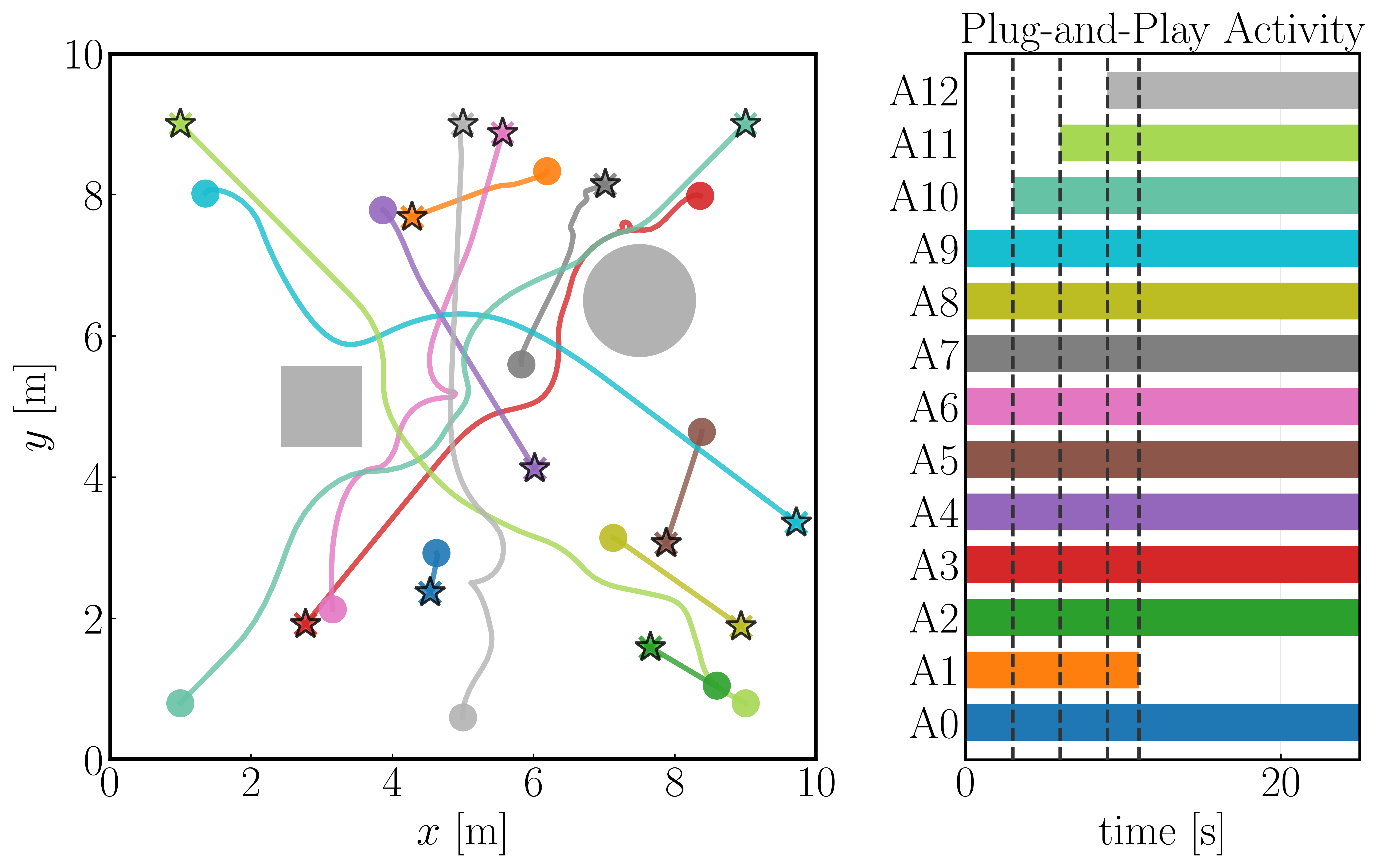}
    \caption{\emph{PnP} scenario with online agent insertion and removal. \textbf{Left:} Closed-loop trajectories in a cluttered workspace (gray obstacles). Filled circles denote initial positions and stars denote final positions; colors correspond to individual agents. \textbf{Right:} Activity timeline indicating when each agent is active/inactive; vertical dashed lines mark join/leave events. The proposed scheme maintains safe, collision-free motion throughout despite the changing set of active agents.}
    \label{fig:PlugAndPlayFigure}
\end{figure}

\vspace*{-0.1cm}
\section{Conclusions and outlook}
\vspace*{-0.1cm}

The proposed decentralized contingency MPC scheme
combines local safe sets with a deterministic
\emph{FoS} update to preserve recursive feasibility
and collision avoidance under exact state observations
and synchronous execution, without inter-agent
communication. Tail-containment constraints preserve
admissibility of shifted contingency plans, while
the contingency-cost decrease condition establishes
Lyapunov-type properties and convergence to a fixed
admissible equilibrium when the selected terminal
equilibrium becomes constant. Nonlinear simulations
illustrate the framework in dense, cluttered, and
\emph{PnP} scenarios.

\vspace{-0.1cm}
Future work includes extending the framework to
time-varying references and establishing robustness
to disturbances, model mismatch, state-estimation
errors, sensing delays, and asynchronous updates.
A memoryless safe-set update could simplify
reconstruction and relax the current \emph{PnP}
joining conditions. Further directions include
runtime optimization, experimental validation,
and comparative evaluation of navigation performance
and computational efficiency. 
\appendix


\appendix
\numberwithin{equation}{section}

\makeatletter
\renewcommand\@seccntformat[1]{%
  \csname the#1\endcsname.\quad
}
\makeatother
\vspace*{-0.1cm}
\section*{Appendix}
\vspace*{-0.1cm}
\section{Auxiliary shifted-tail lemmas}
\vspace*{-0.1cm}

The recursive-feasibility and Lyapunov-type arguments both rely on the same shifted contingency candidate. To avoid duplicate arguments and cross-references between proofs, the two key shifted-tail facts are collected first.

\begin{lemma}[Feasibility of the shifted candidate]
\label{lem:appendix_shifted_candidate_feasible}
Fix $i\in\mathcal I$ and $t\in\mathbb Z_+$ such that
the local FHOCP~\eqref{prob:local_fhocp} is feasible.
Let
\[
\Big(
\{X_i^{\bullet,\ast}(t),U_i^{\bullet,\ast}(t)\}
_{\bullet\in\{\mathrm n,\mathrm c\}},
\bar x_i^{\mathrm c,\ast}(t),
\bar u_i^{\mathrm c,\ast}(t)
\Big)
\]
be an optimal feasible solution. Define, at time $t^+$,
\begin{align}
\tilde x_{i,(k|t^+)}^{\mathrm c}
&:=x_{i,(k+1|t)}^{\mathrm c,\ast},
\qquad k=0,\dots,N_{\mathrm c}-1,
\label{eq:app_shift_x}
\\
\tilde u_{i,(k|t^+)}^{\mathrm c}
&:=u_{i,(k+1|t)}^{\mathrm c,\ast},
\qquad k=0,\dots,N_{\mathrm c}-2,
\label{eq:app_shift_u}
\\
\tilde u_{i,(N_{\mathrm c}-1|t^+)}^{\mathrm c}
&:=\bar u_i^{\mathrm c,\ast}(t),
\label{eq:app_shift_u_terminal}
\\
\tilde{\bar x}_i^{\mathrm c}(t^+)
&:=\bar x_i^{\mathrm c,\ast}(t),
\qquad
\tilde{\bar u}_i^{\mathrm c}(t^+):=\bar u_i^{\mathrm c,\ast}(t).
\label{eq:app_shift_bar}
\end{align}
Then the candidate
\[
\Big(
\{\tilde x_{i,(k|t^+)}^{\mathrm c}\}_{k=0}^{N_{\mathrm c}},
\{\tilde u_{i,(k|t^+)}^{\mathrm c}\}_{k=0}^{N_{\mathrm c}-1},
\tilde{\bar x}_i^{\mathrm c}(t^+),
\tilde{\bar u}_i^{\mathrm c}(t^+)
\Big)
\]
satisfies the contingency-part constraints
\eqref{prob:local_fhocp_init},
\eqref{prob:local_fhocp_dyn_c},
\eqref{prob:local_fhocp_bounds_c},
\eqref{prob:static_obstacle_constraintsMPC},
\eqref{prob:local_fhocp_terminal_state}--\eqref{prob:local_fhocp_terminal_position},
\eqref{prob:local_fhocp_containment},
and \eqref{prob:local_fhocp_tail} at time $t^+$.
\end{lemma}
\vspace{-0.1cm}
\begin{proof}
By feasibility at time $t$,
\begin{equation}
x_{i,(N_{\mathrm c}|t)}^{\mathrm c,\ast}
=
\bar x_i^{\mathrm c,\ast}(t),
\
\bar x_i^{\mathrm c,\ast}(t)
=
f_i\!\bigl(\bar x_i^{\mathrm c,\ast}(t),\bar u_i^{\mathrm c,\ast}(t)\bigr).
\label{eq:app_terminal_eq}
\end{equation}
The applied input is
\(
u_i(t)=u_{i,(0|t)}^{\mathrm n,\ast}=u_{i,(0|t)}^{\mathrm c,\ast}.
\)
Hence, by Assumption~\ref{ass:exact_dynamics},
\begin{equation}
x_i(t^+)
=
f_i\!\bigl(x_{i,(0|t)}^{\mathrm c,\ast},u_{i,(0|t)}^{\mathrm c,\ast}\bigr)
=
x_{i,(1|t)}^{\mathrm c,\ast}.
\label{eq:app_successor_state}
\end{equation}
We first define the terminal candidate state by
\begin{equation}
\tilde x_{i,(N_{\mathrm c}|t^+)}^{\mathrm c}
:=
f_i\!\bigl(
\tilde x_{i,(N_{\mathrm c}-1|t^+)}^{\mathrm c},
\tilde u_{i,(N_{\mathrm c}-1|t^+)}^{\mathrm c}
\bigr).
\label{eq:app_shift_terminal_state}
\end{equation}
Using \eqref{eq:app_shift_x}, \eqref{eq:app_shift_u_terminal}, and \eqref{eq:app_terminal_eq},
\begin{equation}
\begin{aligned}
\tilde x_{i,(N_{\mathrm c}|t^+)}^{\mathrm c}
&=
f_i\bigl(x_{i,(N_{\mathrm c}|t)}^{\mathrm c,\ast},\bar u_i^{\mathrm c,\ast}(t)\bigr)\\
&=
f_i\bigl(\bar x_i^{\mathrm c,\ast}(t),\bar u_i^{\mathrm c,\ast}(t)\bigr)\\
&=
\bar x_i^{\mathrm c,\ast}(t)
=
\tilde{\bar x}_i^{\mathrm c}(t^+).
\label{eq:app_terminal_state_equals_bar}
\end{aligned}
\end{equation}

\noindent\textbf{Initialization and dynamics.}
By \eqref{eq:app_successor_state} and \eqref{eq:app_shift_x},
\[
\tilde x_{i,(0|t^+)}^{\mathrm c}
=
x_{i,(1|t)}^{\mathrm c,\ast}
=
x_i(t^+),
\]
so \eqref{prob:local_fhocp_init} holds for the contingency part.
For $k=0,\dots,N_{\mathrm c}-2$, feasibility at time $t$ gives
\begin{align*}
\tilde x_{i,(k+1|t^+)}^{\mathrm c}
=
x_{i,(k+2|t)}^{\mathrm c,\ast}
&=
f_i\!\bigl(x_{i,(k+1|t)}^{\mathrm c,\ast},u_{i,(k+1|t)}^{\mathrm c,\ast}\bigr)
\\
&=
f_i\!\bigl(\tilde x_{i,(k|t^+)}^{\mathrm c},\tilde u_{i,(k|t^+)}^{\mathrm c}\bigr),
\end{align*}
Thus \eqref{prob:local_fhocp_dyn_c} holds for $k=0,\dots,N_{\mathrm c}-2$, and the case $k=N_{\mathrm c}-1$ is given by \eqref{eq:app_shift_terminal_state}.

\noindent\textbf{State, input, and static-obstacle constraints.}
For $k=0,\dots,N_{\mathrm c}-2$,
\[
\tilde x_{i,(k|t^+)}^{\mathrm c}
=
x_{i,(k+1|t)}^{\mathrm c,\ast}\in\mathcal X_i,
\qquad
\tilde u_{i,(k|t^+)}^{\mathrm c}
=
u_{i,(k+1|t)}^{\mathrm c,\ast}\in\mathcal U_i
\]
by feasibility at time $t$.
For the last stage,
\[
\begin{aligned}
\tilde x_{i,(N_{\mathrm c}-1|t^+)}^{\mathrm c}
&=x_{i,(N_{\mathrm c}|t)}^{\mathrm c,\ast}
=\bar x_i^{\mathrm c,\ast}(t)\in\mathcal X_i,\\
\text{and}\qquad
\tilde u_{i,(N_{\mathrm c}-1|t^+)}^{\mathrm c}
&=\bar u_i^{\mathrm c,\ast}(t)\in\mathcal U_i.
\end{aligned}
\]
The same index shift shows that
\[
h\bigl(\tilde p_{i,(k|t^+)}^{\mathrm c}\bigr)\ge 0,
\qquad \forall k\in\mathbb Z_0^{N_{\mathrm c}},
\]
where $\tilde p_{i,(k|t^+)}^{\mathrm c}:=C_i\tilde x_{i,(k|t^+)}^{\mathrm c}$.

\noindent\textbf{Terminal constraints.}
Equation \eqref{eq:app_terminal_state_equals_bar} yields
\[
\tilde x_{i,(N_{\mathrm c}|t^+)}^{\mathrm c}
=
\tilde{\bar x}_i^{\mathrm c}(t^+),
\]
which is \eqref{prob:local_fhocp_terminal_state}.
Moreover, by \eqref{eq:app_shift_bar} and \eqref{eq:app_terminal_eq},
\[
\tilde{\bar x}_i^{\mathrm c}(t^+)
=
f_i\!\bigl(\tilde{\bar x}_i^{\mathrm c}(t^+),\tilde{\bar u}_i^{\mathrm c}(t^+)\bigr),
\]
so \eqref{prob:local_fhocp_terminal_equilibrium} holds.
Finally, feasibility at time $t$ implies
\[
C_i\bar x_i^{\mathrm c,\ast}(t)\in S_i^\ast(t).
\]
If $\delta(t)=0$ and $\chi_i(t)=1$, then $S_i^\ast(t^+)=S_i^\ast(t)$ by \eqref{eq:freeze_shift_update}, so
\[
C_i\tilde{\bar x}_i^{\mathrm c}(t^+)\in S_i^\ast(t^+).
\]
If $\delta(t)=1$ or $\chi_i(t)=0$, then $S_i^\ast(t^+)=\tilde S_i(t^+)=\Gamma_i(x_i(t^+))$, and the tail-containment argument below with $k=1$ already implies
\[
C_i\tilde{\bar x}_i^{\mathrm c}(t^+)\in S_i^\ast(t^+).
\]
Hence \eqref{prob:local_fhocp_terminal_position} holds.

\noindent\textbf{Active safe-set containment ~\eqref{prob:local_fhocp_containment}.}
It remains to show that the shifted contingency candidate satisfies the active safe-set containment constraint at time~$t^+$.
We distinguish two cases.

\emph{Case 1:} $\delta(t)=0$ and $\chi_i(t)=1$.
Then, by the FoS update rule~\eqref{eq:freeze_shift_update}, the active safe set is frozen, i.e.,
\(
S_i^\ast(t^+) = S_i^\ast(t).
\)
Since the contingency trajectory at time~$t$ was feasible and satisfied~\eqref{prob:local_fhocp_containment} in~$S_i^\ast(t)$, the shifted tail remains contained in the same set.
Hence,~\eqref{prob:local_fhocp_containment} holds at time~$t^+$.

\emph{Case 2:} $\delta(t)=1$ or $\chi_i(t)=0$.
Then the active safe set is updated to the state-induced candidate set,
\(
S_i^\ast(t^+) = \Gamma_i\!\bigl(x_i(t^+)\bigr).
\)
By exact dynamics and the shared first input,
\[
x_i(t^+) = x_{i,(1|t)}^{c,\ast}.
\]
Therefore, the new active safe set is exactly the one induced by the first predicted contingency successor state.
Since the feasible solution at time~$t$ satisfies the tail-containment constraint~\eqref{prob:local_fhocp_tail}, the entire shifted contingency tail is contained in that induced safe set.
Hence,~\eqref{prob:local_fhocp_containment} also holds at time~$t^+$.

\textbf{Tail-containment constraint~\eqref{prob:local_fhocp_tail}.}
Let arbitrary
\(
k\in\mathbb{Z}_{0}^{N_c}, \ l\in\mathbb{Z}_{k}^{N_c}
\)
be given for the candidate at time~$t^+$.
It must be shown that
\[
\bigl\|\tilde p^{\mathrm c}_{i,(l|t^+)}
-
c_i\!\bigl(\tilde x^{\mathrm c}_{i,(k|t^+)}\bigr)\bigr\|
\le
R_i\!\bigl(\tilde x^{\mathrm c}_{i,(k|t^+)}\bigr)-r_i.
\]
For $k\le N_c-1$ and $l\le N_c-1$, the shifted candidate satisfies
\[
\tilde x^{\mathrm c}_{i,(l|t^+)} = x^{\mathrm c,\ast}_{i,(l+1|t)},
\qquad
\tilde x^{\mathrm c}_{i,(k|t^+)} = x^{\mathrm c,\ast}_{i,(k+1|t)}.
\]
Since the safe-set generator is deterministic, this implies
\[
c_i\!\bigl(\tilde x^{\mathrm c}_{i,(k|t^+)}\bigr)
=
c_i\!\bigl(x^{\mathrm c,\ast}_{i,(k+1|t)}\bigr),
\ 
R_i\!\bigl(\tilde x^{\mathrm c}_{i,(k|t^+)}\bigr)
=
R_i\!\bigl(x^{\mathrm c,\ast}_{i,(k+1|t)}\bigr).
\]
Moreover, $l\in\mathbb{Z}_{k}^{N_c-1}$ implies $l+1\in\mathbb{Z}_{k+1}^{N_c}$.
Hence, feasibility of the solution at time~$t$ and constraint~\eqref{prob:local_fhocp_tail} yield
\[
\bigl\|p^{\mathrm c,\ast}_{i,(l+1|t)}
-
c_i\!\bigl(x^{\mathrm c,\ast}_{i,(k+1|t)}\bigr)\bigr\|
\le
R_i\!\bigl(x^{\mathrm c,\ast}_{i,(k+1|t)}\bigr)-r_i,
\]
which is exactly the desired inequality.

For $k\le N_c-1$ and $l=N_c$, we have
\[
\tilde x^{\mathrm c}_{i,(N_c|t^+)}
=
\tilde{\bar x}^{\mathrm c}_i(t^+)
=
\bar x^{\mathrm c,\ast}_i(t)
=
x^{\mathrm c,\ast}_{i,(N_c|t)}.
\]
Since now $N_c\in\mathbb{Z}_{k+1}^{N_c}$, feasibility at time~$t$ again gives
\[
\bigl\|\tilde p^{\mathrm c}_{i,(N_c|t^+)}
-
c_i\!\bigl(\tilde x^{\mathrm c}_{i,(k|t^+)}\bigr)\bigr\|
\le
R_i\!\bigl(\tilde x^{\mathrm c}_{i,(k|t^+)}\bigr)-r_i.
\]

Finally, for $k=N_c$, necessarily $l=N_c$, and the claim follows directly from
\[
\tilde x^{\mathrm c}_{i,(N_c|t^+)} = x^{\mathrm c,\ast}_{i,(N_c|t)}
\]
together with feasibility at time~$t$ for the index pair $(k,l)=(N_c,N_c)$.

Therefore, the shifted candidate satisfies
\eqref{prob:local_fhocp_tail} at time~$t^+$.
\end{proof}
\begin{lemma}[Shifted contingency-cost identity]
\label{lem:appendix_shifted_cost_identity}
Under the hypotheses of Lemma~\ref{lem:appendix_shifted_candidate_feasible} and Assumption~\ref{ass:regularity_cost}(3), the shifted contingency candidate defined in
\eqref{eq:app_shift_x}--\eqref{eq:app_shift_bar} satisfies
\begin{equation}
\tilde J_i^{\mathrm c}(t^+)
=
\hat J_i^{\mathrm c}(t^+),
\label{eq:app_candidate_equals_bound}
\end{equation}
where $\hat J_i^{\mathrm c}(t^+)$ is given by \eqref{eq:lyap_bound_update_main}.
\end{lemma}

\begin{proof}
By definition \eqref{eq:c_cost_def},
\begin{equation}
\begin{aligned}
\tilde J_i^{\mathrm c}(t^+)
=
\sum_{k=0}^{N_{\mathrm c}-1}
&\ell_i^{\mathrm c}\!\Big(
\tilde x_{i,(k|t^+)}^{\mathrm c}-\tilde{\bar x}_i^{\mathrm c}(t^+),
\tilde u_{i,(k|t^+)}^{\mathrm c}-\tilde{\bar u}_i^{\mathrm c}(t^+)
\Big)
\\
&\;+\;
V_i^{\mathrm c}\!\bigl(\tilde{\bar x}_i^{\mathrm c}(t^+),x_i^{\mathrm{ref}}\bigr).
\end{aligned}
\label{eq:app_shifted_candidate_cost}
\end{equation}
For $k=0,\dots,N_{\mathrm c}-2$,
\[
\tilde x_{i,(k|t^+)}^{\mathrm c}-\tilde{\bar x}_i^{\mathrm c}(t^+)
=
x_{i,(k+1|t)}^{\mathrm c,\ast}-\bar x_i^{\mathrm c,\ast}(t),
\]
\[
\tilde u_{i,(k|t^+)}^{\mathrm c}-\tilde{\bar u}_i^{\mathrm c}(t^+)
=
u_{i,(k+1|t)}^{\mathrm c,\ast}-\bar u_i^{\mathrm c,\ast}(t).
\]
For the last stage, \eqref{eq:app_shift_x}, \eqref{eq:app_shift_u_terminal}, and \eqref{eq:app_terminal_eq} imply
\[
\tilde x_{i,(N_{\mathrm c}-1|t^+)}^{\mathrm c}
=
x_{i,(N_{\mathrm c}|t)}^{\mathrm c,\ast}
=
\bar x_i^{\mathrm c,\ast}(t),
\ 
\tilde u_{i,(N_{\mathrm c}-1|t^+)}^{\mathrm c}
=
\bar u_i^{\mathrm c,\ast}(t),
\]
so the last stage contributes
\(
\ell_i^{\mathrm c}(0,0)=0
\)
. Hence
\begin{equation}
\begin{aligned}
\tilde J_i^{\mathrm c}(t^+)
=
\sum_{k=1}^{N_{\mathrm c}-1}
&\ell_i^{\mathrm c}\!\Big(
x_{i,(k|t)}^{\mathrm c,\ast}-\bar x_i^{\mathrm c,\ast}(t),
u_{i,(k|t)}^{\mathrm c,\ast}-\bar u_i^{\mathrm c,\ast}(t)
\Big)
\\
&\;+\;
V_i^{\mathrm c}\!\bigl(\bar x_i^{\mathrm c,\ast}(t),x_i^{\mathrm{ref}}\bigr).
\end{aligned}
\label{eq:app_shifted_cost_identity_pre}
\end{equation}
Comparing \eqref{eq:app_shifted_cost_identity_pre} with the optimal contingency cost at time $t$,
\begin{equation}
\begin{aligned}
J_i^{\mathrm c,\ast}(t)
=
\sum_{k=0}^{N_{\mathrm c}-1}
&\ell_i^{\mathrm c}\!\Big(
x_{i,(k|t)}^{\mathrm c,\ast}-\bar x_i^{\mathrm c,\ast}(t),
u_{i,(k|t)}^{\mathrm c,\ast}-\bar u_i^{\mathrm c,\ast}(t)
\Big)
\\
&\;+\;
V_i^{\mathrm c}\!\bigl(\bar x_i^{\mathrm c,\ast}(t),x_i^{\mathrm{ref}}\bigr),
\end{aligned}
\label{eq:app_opt_cost_t}
\end{equation}
we obtain
\begin{equation}
\tilde J_i^{\mathrm c}(t^+)
=
J_i^{\mathrm c,\ast}(t)
-
\ell_i^{\mathrm c}\!\Big(
x_i(t)-\bar x_i^{\mathrm c,\ast}(t),
u_i(t)-\bar u_i^{\mathrm c,\ast}(t)
\Big),
\label{eq:app_shifted_cost_identity}
\end{equation}
because $x_i(t)=x_{i,(0|t)}^{\mathrm c,\ast}$ and $u_i(t)=u_{i,(0|t)}^{\mathrm c,\ast}$.
By \eqref{eq:lyap_bound_update_main}, the right-hand side of \eqref{eq:app_shifted_cost_identity} is exactly $\hat J_i^{\mathrm c}(t^+)$, which proves Lemma \ref{lem:appendix_shifted_cost_identity}.
\end{proof}

\vspace*{-0.1cm}
\section{Proof of Lemma~\ref{lem:disjoint_invariance}}
\vspace*{-0.1cm}

\begin{proof}
The claim is shown by induction.
Pairwise disjointness at $t=0$ holds by
Assumption~\ref{ass:initial_feasibility}.
Assume that
\[
S_i^\ast(t)\cap S_j^\ast(t)=\emptyset,
\qquad \forall i\neq j,
\]
holds at some time $t\in\mathbb Z_+$.
Fix arbitrary $i\neq j$ and consider the update to $t^+$.
If $\delta(t)=1$, the update rule
\eqref{eq:freeze_shift_update} gives
\[
S_i^\ast(t^+)=\tilde S_i(t^+),
\qquad
S_j^\ast(t^+)=\tilde S_j(t^+).
\]
By definition~\eqref{eq:common_shift_indicator},
the candidate safe sets are pairwise disjoint. Hence,
\[
S_i^\ast(t^+)\cap S_j^\ast(t^+)=\emptyset.
\]
If $\delta(t)=0$, we distinguish the following four cases:
\smallskip
\noindent\textbf{Case 1: $\chi_i(t)=1$ and $\chi_j(t)=1$.}
Then
\(
S_i^\ast(t^+)=S_i^\ast(t),
\
S_j^\ast(t^+)=S_j^\ast(t),
\)
and hence
\[
S_i^\ast(t^+)\cap S_j^\ast(t^+)=\emptyset.
\]
\noindent\textbf{Case 2: $\chi_i(t)=1$ and $\chi_j(t)=0$.}
Then
\(
S_i^\ast(t^+)=S_i^\ast(t),
\
S_j^\ast(t^+)=\tilde S_j(t^+).
\)
Since $\chi_j(t)=0$, definition~\eqref{eq:freezeIndicator}
implies
\[
\tilde S_j(t^+)\cap S_i^\ast(t)=\emptyset.
\]
Therefore,
\[
S_i^\ast(t^+)\cap S_j^\ast(t^+)=\emptyset.
\]
\noindent\textbf{Case 3: $\chi_i(t)=0$ and $\chi_j(t)=1$.}
Symmetric to Case 2.

\noindent\textbf{Case 4: $\chi_i(t)=0$ and $\chi_j(t)=0$.}
Then
\(
S_i^\ast(t^+)=\tilde S_i(t^+),
\
S_j^\ast(t^+)=\tilde S_j(t^+).
\)
Since $\chi_i(t)=0$, definition~\eqref{eq:freezeIndicator}
implies
\[
\tilde S_i(t^+)\cap\tilde S_j(t^+)=\emptyset.
\]
Hence,
\[
S_i^\ast(t^+)\cap S_j^\ast(t^+)=\emptyset.
\]

Since the argument holds for every pair $i\neq j$,
the family $\{S_p^\ast(t^+)\}_{p\in\mathcal I}$
is pairwise disjoint.
This completes the induction.
\end{proof}

\vspace*{-0.1cm}
\section{Proof of Theorem~\ref{thm:rec_feas}}
\vspace*{-0.1cm}

\begin{proof}
We show that feasibility at time $t$ implies feasibility at time $t^+$.
Fix an arbitrary agent $i\in\mathcal I$ and suppose that the local FHOCP~\eqref{prob:local_fhocp} is feasible at time $t$.

By Lemma~\ref{lem:appendix_shifted_candidate_feasible}, the shifted contingency candidate defined by
\eqref{eq:app_shift_x}--\eqref{eq:app_shift_bar} satisfies all contingency-part constraints of \eqref{prob:local_fhocp} at time $t^+$.
By Lemma~\ref{lem:appendix_shifted_cost_identity}, its contingency cost satisfies
\[
\tilde J_i^{\mathrm c}(t^+)=\hat J_i^{\mathrm c}(t^+),
\]
so the Lyapunov-type constraint \eqref{prob:local_fhocp_lyap} is also feasible at time $t^+$.

Hence a feasible contingency candidate exists at time $t^+$.
By Assumption~\ref{ass:nominal_completion}, this feasible contingency candidate admits a feasible nominal completion satisfying the remaining constraints of the local FHOCP, including the shared-first-input constraint \eqref{prob:local_fhocp_shared}.
Therefore, the full local FHOCP~\eqref{prob:local_fhocp} is feasible at time $t^+$.

Since feasibility holds at $t=0$ by Assumption~\ref{ass:initial_feasibility}, induction yields feasibility for all $t\in\mathbb Z_+$.
\end{proof}

\vspace*{-0.1cm}
\section{Proof of Theorem~\ref{thm:safety}}
\vspace*{-0.1cm}

\begin{proof}
By Theorem~\ref{thm:rec_feas}, the local FHOCP remains feasible for every agent and every time.
Therefore the safe-set containment constraint \eqref{prob:local_fhocp_containment} holds along the closed loop, and thus
\(
\mathcal B_i(t)\subseteq S_i^\ast(t),
\ \forall i\in\mathcal I,\ \forall t\in\mathbb Z_+.
\)
By Lemma~\ref{lem:disjoint_invariance}, the active safe sets remain pairwise disjoint:
\[
S_i^\ast(t)\cap S_j^\ast(t)=\emptyset,
\qquad \forall i\neq j,\ \forall t\in\mathbb Z_+.
\]
Hence, for any distinct agents $i\neq j$,
\[
\mathcal B_i(t)\subseteq S_i^\ast(t),
\qquad
\mathcal B_j(t)\subseteq S_j^\ast(t),
\]
it holds that
\[
\mathcal B_i(t)\cap \mathcal B_j(t)=\emptyset \qquad \forall i\neq j,\ \forall t\in\mathbb Z_+.
\]
\end{proof}

\vspace*{-0.1cm}
\section{Proof of Proposition~\ref{prop:lyap_decrease}}
\vspace*{-0.1cm}

\begin{proof}
Fix an arbitrary agent $i\in\mathcal I$ and time $t\in\mathbb Z_+$.
By recursive feasibility, the local FHOCP~\eqref{prob:local_fhocp} is feasible at both $t$ and $t^+$.
By Lemma~\ref{lem:appendix_shifted_candidate_feasible}, the shifted contingency candidate is feasible at time $t^+$, and by Lemma~\ref{lem:appendix_shifted_cost_identity},
\[
\tilde J_i^{\mathrm c}(t^+)=\hat J_i^{\mathrm c}(t^+).
\]
Therefore, feasibility of the solution at time $t^+$ gives
\begin{equation}
J_i^{\mathrm c,\ast}(t^+)
\le
\tilde J_i^{\mathrm c}(t^+)
=
\hat J_i^{\mathrm c}(t^+).
\label{eq:app_optimality_upper_bound}
\end{equation}
Substituting the bound update \eqref{eq:lyap_bound_update_main} into \eqref{eq:app_optimality_upper_bound} yields
\[
J_i^{\mathrm c,\ast}(t^+)
\le
J_i^{\mathrm c,\ast}(t)
-
\ell_i^{\mathrm c}\!\bigl(e_i(t),\delta u_i(t)\bigr),
\]
which is exactly \eqref{eq:lyap_decrease_main}.
Since $\ell_i^{\mathrm c}\ge 0$, the sequence $\{J_i^{\mathrm c,\ast}(t)\}_{t\in\mathbb Z_+}$ is monotonically nonincreasing.
Moreover, $J_i^{\mathrm c,\ast}(t)\ge 0$ for all $t$, because \eqref{eq:c_cost_def} is the sum of the nonnegative stage cost and the nonnegative offset cost.
Hence $\{J_i^{\mathrm c,\ast}(t)\}$ is lower bounded and therefore convergent.
Summing \eqref{eq:lyap_decrease_main} from $t=0$ to $T-1$ gives
\begin{equation}
\sum_{t=0}^{T-1}
\ell_i^{\mathrm c}\!\bigl(e_i(t),\delta u_i(t)\bigr)
\le
J_i^{\mathrm c,\ast}(0)-J_i^{\mathrm c,\ast}(T)
\le
J_i^{\mathrm c,\ast}(0).
\label{eq:app_telescoping_sum}
\end{equation}
Letting $T\to\infty$ yields \eqref{eq:summable_stage_main}.
In particular, the nonnegative sequence
$\ell_i^{\mathrm c}\!\bigl(e_i(t),\delta u_i(t)\bigr)$ converges to zero.

Finally, by Assumption~\ref{ass:regularity_cost}(3), there exists
$\underline{\alpha}_{\ell,i}\in\mathcal K_\infty$ such that
\[
\underline{\alpha}_{\ell,i}\!\bigl(\|(e_i(t),\delta u_i(t))\|\bigr)
\le
\ell_i^{\mathrm c}\!\bigl(e_i(t),\delta u_i(t)\bigr),
\qquad \forall t\in\mathbb Z_+.
\]
Suppose, for contradiction, that $\|(e_i(t),\delta u_i(t))\|\not\to 0$.
Then there exist $\varepsilon>0$ and an infinite subsequence $\{t_m\}_{m\in\mathbb Z_+}$ such that
\[
\|(e_i(t_m),\delta u_i(t_m))\|\ge \varepsilon,
\qquad \forall m\in\mathbb Z_+.
\]
By strict monotonicity of $\underline{\alpha}_{\ell,i}$,
\[
\ell_i^{\mathrm c}\!\bigl(e_i(t_m),\delta u_i(t_m)\bigr)
\ge
\underline{\alpha}_{\ell,i}(\varepsilon)
>0,
\qquad \forall m\in\mathbb Z_+,
\]
which contradicts
\(
\ell_i^{\mathrm c}\!\bigl(e_i(t),\delta u_i(t)\bigr)\to 0.
\)
Therefore,
\(
\|(e_i(t),\delta u_i(t))\|\to 0
\qquad\text{as } t\to\infty.
\)
\end{proof}

\vspace*{-0.1cm}
\section{Proof of Corollary~\ref{cor:fixed_ss_convergence}}
\vspace*{-0.1cm}

\begin{proof}
By Proposition~\ref{prop:lyap_decrease},
\[
\|(e_i(t),\delta u_i(t))\|\to 0.
\]
Assume that there exists $\bar t\in\mathbb Z_+$ such that
\[
\bar x_i^{\mathrm c,\ast}(t)=\bar x_i^{\mathrm c,\star},
\qquad
\bar u_i^{\mathrm c,\ast}(t)=\bar u_i^{\mathrm c,\star},
\qquad \forall t\ge \bar t.
\]
Then, $\forall t\ge \bar t$,
\(
e_i(t)=x_i(t)-\bar x_i^{\mathrm c,\star},
\
\delta u_i(t)=u_i(t)-\bar u_i^{\mathrm c,\star}.
\)
Hence
\[
x_i(t)\to \bar x_i^{\mathrm c,\star},
\qquad
u_i(t)\to \bar u_i^{\mathrm c,\star}
\qquad\text{as } t\to\infty.
\]
\end{proof}

\vspace*{-0.1cm}
\section{Proof of Corollary~\ref{cor:pnp_preservation}}
\vspace*{-0.1cm}

\begin{proof}
\noindent\textbf{Join.}
By Assumption~\ref{ass:pp_join_feasible}, the joining agent can reconstruct all active safe sets and
initializes a feasible local FHOCP with a disjoint safe set.
Feasibility of the existing agents is unchanged.
Thus, initial feasibility and pairwise separation hold
for $\mathcal I^+$, and Theorems~\ref{thm:rec_feas}
and~\ref{thm:safety} apply.

\vspace{-3mm}
\noindent\textbf{Leave.}
Removing an agent preserves feasibility and pairwise
separation for the remaining agents. Hence,
Theorems~\ref{thm:rec_feas} and~\ref{thm:safety}
apply to $\mathcal I^-$.
\end{proof}

\bibliographystyle{plain}
\bibliography{autosam}

\end{document}